\documentclass[11pt]{amsart}
\usepackage{amsmath,amssymb,mathrsfs,bm,enumerate}
\usepackage{pstricks}
\usepackage[pagebackref, colorlinks=true, linkcolor=black, citecolor=black, urlcolor=black]{hyperref}
\usepackage{tikz-cd} 

\usepackage[enableskew]{youngtab} 
\usepackage[boxsize=0.7em]{ytableau} 

\usepackage{stmaryrd} 

\newcommand{\excise}[1]{}

\psset{unit=1pt}
\psset{arrowsize=4pt 1}
\psset{linewidth=.5pt}

\newgray{lightgray}{.80}

\countdef\x=23
\countdef\y=24
\countdef\z=25
\countdef\t=26

\def\graybox(#1,#2){
\x=#1 \y=#2 
\z=\x \t=\y
\advance\z by 10 
\advance\t by 10 
\psframe[fillstyle=solid,fillcolor=lightgray,linewidth=0pt](\x,\y)(\z,\t) 
\psline[linewidth=.5pt](\x,\y)(\x,\t)(\z,\t)(\z,\y)(\x,\y)}

\def\emptygraybox(#1,#2){
\x=#1 \y=#2 
\z=\x \t=\y
\advance\z by 10 
\advance\t by 10 
\psframe[fillstyle=solid,fillcolor=lightgray,linewidth=0pt,linecolor=lightgray](\x,\y)(\z,\t)}

\def\blankbox(#1,#2){
\x=#1 \y=#2 
\z=\x \t=\y
\advance\z by 10 
\advance\t by 10 
\psframe[linewidth=.5pt](\x,\y)(\z,\t)}

\def\whitebox(#1,#2){
\x=#1 \y=#2 
\z=\x \t=\y
\advance\z by 10 
\advance\t by 10 
\psframe[fillstyle=solid,fillcolor=white,linewidth=0pt](\x,\y)(\z,\t) 
\psline[linewidth=.5pt](\x,\y)(\x,\t)(\z,\t)(\z,\y)(\x,\y)}

\def\whiteboxb(#1,#2){
\x=#1 \y=#2 
\z=\x \t=\y
\advance\z by 10 
\advance\t by 10 
\psframe[fillstyle=solid,fillcolor=white,linewidth=0pt](\x,\y)(\z,\t)}

\newcommand{\isom}{\cong}

\newcommand{\exterior}{\textstyle\bigwedge}
\renewcommand{\tilde}{\widetilde}
\renewcommand{\hat}{\widehat}
\renewcommand{\bar}{\overline}

\newcommand{\ZZ}{\mathbb{Z}}
\newcommand{\PP}{\mathbb{P}}
\newcommand{\bS}{\mathbb{S}}

\newcommand{\OO}{\mathcal{O}}

\newcommand{\Fl}{{Fl}}
\newcommand{\Gr}{{Gr}}

\newcommand{\bFl}{\mathbf{Fl}}
\newcommand{\bGr}{\mathbf{Gr}}

\newcommand{\ee}{\mathrm{e}}

\newcommand{\cL}{\mathcal{L}}
\newcommand{\cE}{\mathcal{E}}
\newcommand{\cF}{\mathcal{F}}
\newcommand{\cV}{\mathcal{V}}

\newcommand{\ba}{\bm{a}}
\newcommand{\bb}{\bm{b}}
\newcommand{\bp}{\bm{p}}
\newcommand{\bq}{\bm{q}}

\newcommand{\bOmega}{\bm{\Omega}}
\newcommand{\bW}{\bm{W}}

\newcommand{\pr}{\pi_1}
\newcommand{\prr}{{\pi_2}}

\DeclareMathOperator{\codim}{codim}
\DeclareMathOperator{\rk}{rk}

\DeclareMathOperator{\Pic}{Pic}

\DeclareMathOperator{\ch}{ch}
\DeclareMathOperator{\sgn}{sgn}

\let\originalleft\left
\let\originalright\right
\renewcommand{\left}{\mathopen{}\mathclose\bgroup\originalleft}
\renewcommand{\right}{\aftergroup\egroup\originalright}

\newtheorem{thm}{Theorem}

\newtheorem{theorem}{Theorem}[section]
\newtheorem{lemma}[theorem]{Lemma}
\newtheorem{proposition}[theorem]{Proposition}
\newtheorem{corollary}[theorem]{Corollary}

\newtheorem*{thm*}{Theorem}
\newtheorem*{lem*}{Lemma}
\newtheorem*{prop*}{Proposition}
\newtheorem*{cor*}{Corollary}

\theoremstyle{definition}

\newtheorem{remark}[theorem]{Remark}
\newtheorem{example}[theorem]{Example}

\newtheorem*{defn*}{Definition}
\newtheorem*{rmk*}{Remark}

\font\co=lcircle10
\def\petit#1{{\scriptstyle #1}}

\def\jr{\smash{\raise2pt\hbox{\co \rlap{\rlap{\char'005} \char'007}}
               \raise6pt\hbox{\rlap{\vrule height5pt}}
               \raise2pt\hbox{\rlap{\hskip4pt \vrule height0.4pt depth0pt
                width5.7pt}}
               \raise2pt\hbox{\rlap{\hskip-9.5pt \vrule height.4pt depth0pt
                width6.2pt}}
               \lower6pt\hbox{\rlap{\vrule height4.5pt}}}}
\def\rj{\smash{\raise2pt\hbox{\co \rlap{\rlap{\char'004} \char'006}}
               \raise6pt\hbox{\rlap{\vrule height5pt}}
               \raise2pt\hbox{\rlap{\hskip4pt \vrule height0.4pt depth0pt
                width5.7pt}}
               \raise2pt\hbox{\rlap{\hskip-9.5pt \vrule height.4pt depth0pt
                width6.2pt}}
               \lower6pt\hbox{\rlap{\vrule height4.5pt}}}}
\def\je{\smash{\raise2pt\hbox{\co \rlap{\rlap{\char'005}
                \phantom{\char'007}}}\raise6pt\hbox{\rlap{\vrule height5pt}}
               \raise2pt\hbox{\rlap{\hskip-9.5pt \vrule height.4pt depth0pt
                width6.2pt}}}}
\def\ej{\smash{\raise2pt\hbox{\co \rlap{\rlap{\char'004}\phantom{\char'006}}}
               \raise2pt\hbox{\rlap{\hskip-9.5pt \vrule height.4pt depth0pt
                width6.2pt}}
               \lower6pt\hbox{\rlap{\vrule height4.5pt}}}}
\def\er{\smash{\raise2pt\hbox{\co \rlap{\rlap{\phantom{\char'005}} \char'007}}
               \raise2pt\hbox{\rlap{\hskip4pt \vrule height0.4pt depth0pt
                width5.7pt}}
               \lower6pt\hbox{\rlap{\vrule height4.5pt}}}}
\def\re{\smash{\raise2pt\hbox{\co \rlap{\rlap{\phantom{\char'004}} \char'006}}
               \raise6pt\hbox{\rlap{\vrule height5pt}}
               \raise2pt\hbox{\rlap{\hskip4pt \vrule height0.4pt depth0pt
                width5.7pt}}}}
\def\+{\smash{\lower6pt\hbox{\rlap{\vrule height17pt}}
                \raise2pt
                \hbox{\rlap{\hskip-9pt \vrule height.4pt depth0pt
                width18.7pt}}}}
\def\hor{\smash{\raise2pt\hbox{\rlap{\hskip-9.5pt \vrule height.4pt depth0pt
                width19.2pt}}}}
\def\ver{\smash{\lower6pt\hbox{\rlap{\vrule height17pt}}}}
\def\ho{\smash{\hbox{\rlap{\vrule height5pt}}
                \raise2pt
                \hbox{\rlap{\hskip-9pt \vrule height.4pt depth0pt
                width18.7pt}}}}

\def\perm#1#2{\hbox{\rlap{$\petit {#1}_{\scriptscriptstyle #2}$}}%
                \phantom{\petit 1}}

\def\textcross{\ \smash{\lower4pt\hbox{\rlap{\hskip4.15pt\vrule height14pt}}
                \raise2.8pt\hbox{\rlap{\hskip-3pt \vrule height.4pt depth0pt
                width14.7pt}}}\hskip12.7pt}
\def\textelbow{\ \hskip.1pt\smash{\raise2.75pt%
               \hbox{\co \hskip 4.15pt\rlap{\rlap{\char'005} \char'007}
                \lower6.8pt\rlap{\vrule height3.5pt}
                \raise3.6pt\rlap{\vrule height3.5pt}}
                \raise2.8pt\hbox{%
                  \rlap{\hskip-7.15pt \vrule height.4pt depth0pt width3.5pt}%
                  \rlap{\hskip4.05pt \vrule height.4pt depth0pt width3.5pt}}}
                \hskip8.7pt}

\begin{document}

\title[K-classes of Brill-Noether loci and a determinantal formula]{K-classes of Brill-Noether loci\\ and a determinantal formula}


\author{Dave Anderson}
\address{Dave Anderson
\newline \indent Department of Mathematics, The Ohio State University  
\newline \indent Columbus, OH 43210}
\email{anderson.2804@math.osu.edu}

\author{Linda Chen}
\address{Linda Chen
\newline \indent Department of Mathematics and Statistics, Swarthmore College 
\newline \indent  Swarthmore, PA 19081}
\email{lchen@swarthmore.edu}

\author{Nicola Tarasca}
\address{Nicola Tarasca
\newline \indent Department of Mathematics \& Applied Mathematics
\newline \indent Virginia Commonwealth University, Richmond, VA 23284}
\email{tarascan@vcu.edu}

\thanks{DA was partially supported by NSF Grant DMS-1502201.}

\subjclass[2010]{14H51, 14M15 (primary),  19E20, 05E05 (secondary)}
\keywords{Brill-Noether loci, K-theory classes, determinantal formulas, Young diagrams, \linebreak \indent double Schubert and Grothendieck polynomials}

\begin{abstract}
We compute the Euler characteristic of the structure sheaf of the Brill-Noether locus of linear series with special vanishing at up to two marked points.  
When the Brill-Noether number $\rho$ is zero, we recover the Castelnuovo formula for the number of special linear series on a general curve; when $\rho=1$, we recover the formulas of Eisenbud-Harris, Pirola, and Chan-Mart\'in-Pflueger-Teixidor for the arithmetic genus of a Brill-Noether curve of special divisors.

These computations are obtained as applications of a new determinantal formula for the K-theory class of certain degeneracy loci.  
Our degeneracy locus formula also specializes to new determinantal expressions for the double Grothendieck polynomials corresponding to 321-avoiding permutations, and gives double versions of the {flagged skew Grothendieck polynomials} recently introduced by Matsumura.  Our result extends the formula of Billey-Jockusch-Stanley expressing Schubert polynomials for 321-avoiding permutations as generating functions for flagged skew tableaux.
\end{abstract}

\maketitle



Given a smooth projective curve $C$ of genus $g$ over an algebraically closed field, the classical {\it Brill-Noether theorem} describes the locus of special line bundles
\[
  W^r_d(C) = \left\{L \in \Pic^d(C) \,|\, h^0(C,L)\geq r+1 \right\}.
\]
A parameter count --- reviewed at the end of this introduction --- estimates the dimension of $W^r_d(C)$ as $\rho = \rho(g,r,d) := g - (r+1)(g-d+r)$, and the Brill-Noether theorem states that when $C$ has general moduli, the locus $W^r_d(C)$ is in fact nonempty of dimension $\rho$ whenever $\rho\geq 0$.  
A connection with degeneracy loci for maps of vector bundles was implicit in the original work by Brill and Noether, and was brought into focus by Kleiman and Laksov in one of the several modern proofs of the theorem given in the 1970s.

In this article we prove two main theorems.  The first gives a formula for the holomorphic Euler characteristic (that is, the arithmetic genus) 
of the Brill-Noether locus --- and in fact, for the generalized Brill-Noether loci parametrizing linear series having specific vanishing profiles at one or two points.  Our results extend the classical computation by Castelnuovo, who studied the zero-dimensional case $\rho=$~$0$; Eisenbud-Harris \cite{eh} and Pirola \cite{p}, who studied the case $\rho=1$; and Chan-Mart\'in-Pflueger-Teixidor \cite{clpt}, whose remarkable computation uses the combinatorics of tableaux and the geometry of limit linear series to treat the case when the two-pointed locus is one-dimensional.  As a by-product of our formulas, we obtain a new proof of an existence criterion for special linear series, originally due to Osserman.

Our genus formulas are deduced from the second main theorem of the article: a new \textit{determinantal} formula for the K-theory class of a certain type of degeneracy loci.  These loci arise naturally not only from the Brill-Noether problem, but also in combinatorics --- they are built from a  class of permutations called {\it 321-avoiding permutations}.  As another application of our degeneracy locus formula, we find new determinantal formulas for families of polynomials occurring in algebraic combinatorics known as the {\it double Schubert} and {\it double Grothendieck polynomials}.  These results extend recent work of Matsumura \cite{matsumura}, Hudson-Matsumura \cite{hm}, and Hudson-Ikeda-Matsumura-Naruse \cite{himn, himn2}.

Another goal of this work is to highlight the connection between recent developments in Schubert calculus and the geometry of curves.  The results of this paper expand on the fruitful interactions which led to the growth of both subjects, as discussed extensively  in \cite{acgh}.  On one hand, an approach to linear series via degeneracy loci unifies, and perhaps simplifies, several results in Brill-Noether theory --- for example, one may compute the Euler characteristic of a one-pointed Brill-Noether locus by applying the determinantal formula of \cite{himn}.  On the other hand, constructions arising in the study of linear series led us to the geometric proof of the general determinantal formula presented in \S\ref{determinantalF}.  It seems natural to expect that further progress can be made in both subjects by exploiting this bridge.

We now turn to more precise statements of the main results.  The locus $W^r_d(C)$ of special line bundles on a smooth curve $C$ has a canonical desingularization by the variety of linear series $G^r_d(C)$, which parametrizes pairs $\ell=(L,V)$ with $L\in\Pic^d(C)$ and $V\subseteq H^0(C,L)$ an $(r+1)$-dimensional subspace.  For a given linear series $\ell$ and a point $P\in C$, the {\it vanishing sequence} of $\ell$ at $P$ is the sequence
\[
  \ba^\ell(P) = \left(0\leq a^\ell_0(P)<a^\ell_1(P)<\cdots< a^\ell_r(P)\leq d \right)
\]
of distinct orders of vanishing of sections in $V$ at $P$.

The {\it two-pointed Brill-Noether locus} is defined as follows.  Fix two  points $P$ and $Q$ on a smooth curve $C$.  Given sequences of integers
\begin{align*}
  \bm{a} &= (0\leq a_0<a_1<\cdots <a_{r}\leq d) \quad \text{and} \\
  \bm{b} &= (0\leq b_0<b_1<\cdots <b_{r}\leq d),
\end{align*}
we wish to parametrize linear series $\ell$ of projective dimension $r$ and degree $d$ on $C$ with $\ba^\ell(P)$ dominating $\ba$, and $\ba^\ell(Q)$ dominating $\bb$.  That is,
\begin{align*}
 G^{\bm{a},\bm{b}}_d(C,P,Q) := \left\{ \ell \in G^r_d(C)\,|\, a^\ell_i(P) \geq a_i \text{ and } a^\ell_i(Q) \geq b_i \text{ for all } 0\leq i\leq r \right\}.
\end{align*}

We will require the following nontrivial fact about curves as input.
The {\em two-pointed Brill-Noether theorem} says that
for a \textit{general} two-pointed curve $(C,P,Q)$ of genus $g$, the Brill-Noether locus $G^{\bm{a},\bm{b}}_d(C,P,Q)$ is either empty or has dimension equal to the {\em two-pointed Brill-Noether number}:
\[
  \rho:=\rho(g,r,d,\bm{a},\bm{b}) 
    = g - \sum_{i=0}^{r} (g-d + a_{i}+b_{r-i}). 
\]
This was first proved by Eisenbud and Harris using limit linear series and a construction on a singular curve \cite[\S1]{eh}.  More recently, explicit examples of smooth two-pointed curves satisfying the two-pointed Brill-Noether theorem in any genus have been constructed, by studying curves on decomposable elliptic ruled surfaces \cite[\S 2]{ft}.  In contrast to the situation with $G^r_d(C)$, the condition $\rho\geq 0$ is not sufficient to guarantee that the pointed locus $G^{\bm{a},\bm{b}}_d(C,P,Q)$ is nonempty.  A numerical criterion for nonemptiness was given by Osserman \cite{o}, and also follows from our results, see Proposition \ref{non-emptiness}.

Our first main theorem computes the holomorphic (sheaf) Euler characteristic of the locus $G^{\bm{a},\bm{b}}_d(C,P,Q)$ when this has expected dimension $\rho$.  To state it, we need some more notation.  Given sequences $\bm{a}$ and $\bm{b}$ as above, we define two partitions $\lambda$ and $\mu$ by setting
\begin{align*}
 \lambda_i &= n+ a_{r+1-i}-(r+1-i)  , \;\text{ and }\\
 \mu_i &= n-b_{i-1}+i-1 -g+d-r
\end{align*}
for $1\leq i\leq r+1$, where $n$ is a fixed, sufficiently large nonnegative integer.

Partitions are commonly represented as {\it Young diagrams}, so $\lambda$ is a collection of boxes with $\lambda_i$ boxes in the $i$-th row.  When $\mu_i\leq \lambda_i$ for all $i$, one has $\mu\subseteq \lambda$, and one represents the sequence $\lambda_i-\mu_i$ as a {\it skew Young diagram} $\lambda/\mu$ (the complement of $\mu$ in $\lambda$).  Borrowing this notation, we will write $|l/m| = \sum_{i=1}^{r+1} (l_i-m_i)$ for any sequences of integers $l$ and $m$ of length $r+1$, regardless of whether $l_i-m_i\geq 0$.

\begin{thm}\label{t.BN}
Let $(C,P,Q)$ be a smooth two-pointed curve of genus $g$.
If $G:=G^{\bm{a},\bm{b}}_d(C,P,Q)$ has dimension equal to $\rho$, then its Euler characteristic is
\begin{align}
\label{e.t-main}
  \chi\left(\OO_{G}\right)= \sum_{ l, m } \left( \prod_{i=1}^{r+1} \binom{\mu_i}{\mu_i-m_i}\binom{-\lambda_i}{l_i-\lambda_i} \right) g! \left| \frac{1}{(l_i-m_j +j-i)!} \right|_{1\leq i,j\leq r+1}
\end{align}
the sum being taken over all nonnegative integer sequences $l$ and $m$ such that $m_i \leq   \mu_i$ and $l_i \geq  \lambda_i$ for all $i$, and such that $|l/m| = |\lambda/\mu|+\rho$.
\end{thm}

The proof is given in \S\ref{Euler}.  
In the statement, the binomial coefficients for a negative integer $-s$ are given by $ \binom{-s}{k}=\frac{-s(-s-1)\cdots(-s-k+1)}{k!}=(-1)^k\binom{s+k-1}{k}$, for $k\geq 0$.
Also, the sequences $l$ and $m$ need not be partitions, and even when they are, $l/m$ need not be a skew Young diagram --- indeed, $\lambda/\mu$ itself may not be skew. However, with a more detailed combinatorial analysis, one can rewrite the formula so that terms where $l$ and $m$ are partitions are the only ones which contribute to the sum --- see Theorem~\ref{t.euler-two-pointed-tableau}.

\medskip

We now turn to the degeneracy locus formulas.  Hudson-Ikeda-Matsu\-mu\-ra-Naruse gave a determinantal formula for the K-theory class of the structure sheaf of a Schubert variety in a Grassmann bundle \cite{himn}, and this formula may be applied to obtain the one-pointed case of Theorem~\ref{t.BN}.  The formula of \cite{himn} was subsequently refined in \cite{a} and \cite{hm}, but the loci considered by these authors are not sufficient to compute the class of a two-pointed Brill-Noether variety --- so we require a new determinantal formula in K-theory, which is of independent interest.  A special case of our formula is related algebraically to a formula of Matsumura \cite{matsumura}.

Here is the general setup.  Given decreasing sequences of integers $\bp$ and $\bq$, consider vector bundles
\[
E_{p_t}\hookrightarrow \cdots \hookrightarrow E_{p_1} \xrightarrow{\varphi} F_{q_1} \twoheadrightarrow \cdots \twoheadrightarrow F_{q_t}
\]
on a nonsingular variety $X$, with the ranks indicated by subscripts.  The degeneracy locus is
\[
   W_{\bp,\bq} = \left\{ x\in X\,|\, \dim\ker( E_{p_j} \to F_{q_i}) \geq 1+i-j \text{ for all }i,j\right\}.
\]
From the data $\bp,\bq$, we define partitions $\lambda$ and $\mu$ by
\begin{align*}
  \lambda_i &= q_i-t+i, &
  \mu_j &= p_{j}-(t+1-j).
\end{align*}
These partitions are related to the ones associated to the Brill-Noether loci; see the discussion at the end of this introduction for a special case, and \S\ref{GandW} for more detail.

In order for the rank conditions defining $W_{\bp,\bq}$ to be feasible and nontrivial, one should require $\lambda_i\geq \mu_i$, so that $\lambda/\mu$ forms a skew Young diagram.  The expected codimension of the locus $W_{\bp,\bq}$  equals  $|\lambda/\mu|$.

We compute the class of $W_{\bp,\bq}$ as a variation of a skew Schur determinant.  Given partitions $\lambda = (\lambda_1\geq \cdots \geq \lambda_t)$ and $\mu = (\mu_1\geq \cdots \geq \mu_t)$, and doubly indexed series $c(i,j)=\sum_{m\geq 0}c_m(i,j)$, for $1\leq i,j\leq t$, let us define the determinant
\[
  \Delta_{\lambda/\mu}(c;\beta) := \left| \sum_{k\geq 0} \binom{\lambda_i-\mu_j+k-1}{k}\beta^k c_{\lambda_i-\mu_j+j-i+k}(i,j) \right|_{1\leq i,j\leq t}.
\]
The notation for the entries of this determinant can be condensed by using the operator $T$ which raises the index of $c(i,j)$, so $T^k\cdot c_m(i,j) = c_{m+k}(i,j)$.  Then
\[
  \Delta_{\lambda/\mu}(c;\beta) = \left| (1-\beta T)^{-\lambda_i+\mu_j} c_{\lambda_i-\mu_j+j-i}(i,j) \right|_{1\leq i,j\leq t}.
\]
When $\beta=0$ and $c(i,j) = \prod_{k=1}^n (1-x_k)^{-1}$ for all $i,j$, this is the classical Jacobi-Trudi formula for the skew Schur function $s_{\lambda/\mu}(x)$.

\begin{thm}\label{t.degloci}
Assume that $\lambda_i-\mu_i\geq 0$ for all $i$, and that $W:=W_{\bp,\bq}$ has codimension $|\lambda/\mu|$.  The class of $W$ in the Grothendieck group of coherent sheaves $K_\circ(X)$ is
\[
  \left[\OO_W\right] = \Delta_{\lambda/\mu}( c;-1 )\cdot [\OO_X],
\]
where $c(i,j) = c^K(F_{q_i}-E_{p_{j}})$ is the K-theoretic Chern class.
\end{thm}

This is proved in \S\ref{determinantalF}, as part (\ref{t.deg-part2}) of Theorem~\ref{t.deg}.  Part (\ref{t.deg-part1}) of Theorem~\ref{t.deg} provides a more general statement needed for the proof of Theorem~\ref{t.BN}.  In fact, all the formulas we prove take place in the {\em connective K-theory} of $X$, a module over $\ZZ[\beta]$ which interpolates K-theory (at $\beta=-1$) and Chow groups (at $\beta=0$).  So our formulas also specialize directly to cohomology.  Moreover, in Theorem~\ref{t.deg} we remove the assumption that $X$ is smooth, allowing rational singularities.

There is a general correspondence between degeneracy loci and permutations, as explained in \cite{fulton-flags}, for example.  Our loci $W_{\bp,\bq}$ are exactly those corresponding to {\it 321-avoiding permutations}, i.e., permutations with no decreasing subsequence of length three.  Under this correspondence, the formulas for general degeneracy loci are related to the (double) Schubert polynomials and Grothendieck polynomials of Lascoux and Sch\"utzenberger.  Our K-theoretic results therefore give new determinantal formulas for the double Grothendieck polynomials of 321-avoiding permutations, extending work by Matsumura \cite{matsumura}.  Specializing to cohomology, we recover formulas of Billey-Jockusch-Stanley \cite{bjs} and Chen-Li-Louck \cite{cll}, giving new proofs via geometry.  The details, including the correspondence between $(\bp,\bq)$ and 321-avoiding permutations, are described in \S\ref{Schub}.

In \S\ref{Exs}, we explain how our results can be phrased in terms of the combinatorics of tableaux.  A \emph{row semi-standard Young tableau} on a skew diagram $\lambda/\mu$ is a filling of the boxes of $\lambda/\mu$ whose entries are strictly increasing along rows and weakly decreasing down columns.  A \emph{strict Young tableau} is a filling whose entries are strictly increasing across each row and down each column.  A \emph{standard Young tableau} is a strict Young tableau using the numbers $1,\ldots,|\lambda/\mu|$.

The number of standard Young tableaux  on a skew shape $\lambda/\mu$ is denoted by $f^{\lambda/\mu}$.  We will use $\alpha^{\lambda/\mu}$ to denote the number of row semi-standard Young tableaux on $\lambda/\mu$ whose entries in row $i$ are in $\{1,\ldots,\lambda_i\}$; and $\zeta^{\lambda/\mu}$ for the number of strict Young tableaux whose entries in row $i$ are in $\{1,\ldots,\lambda_i-1\}$.

\begin{thm}
\label{t.euler-two-pointed-tableau} If $\dim G^{\bm{a},\bm{b}}_d(C,P,Q)=\rho$, then the Euler characteristic is
\begin{align*}
\chi\left(\mathcal{O}_{G^{\bm{a},\bm{b}}_d(C,P,Q)}\right) 
&=\sum_{\lambda^+,\mu^-} (-1)^{|\lambda^+/\lambda|} \cdot \alpha^{\mu/\mu^-} \cdot \zeta^{\lambda^+/\lambda} \cdot f^{\lambda^+/\mu^-}
\end{align*}
where the sum is  over partitions $\mu^-\subseteq \mu$ and $\lambda^+\supseteq \lambda$ of length $r+1$ such that $|\lambda^+/\mu^-|=|\lambda/\mu|+\rho$. 
\end{thm}

Special cases of  Theorem~\ref{t.euler-two-pointed-tableau} include Castelnuovo's formula and the Eisen\-bud-Harris-Pirola formula.  Its proof is given in \S\ref{Exs}, along with a discussion of other special cases and further connections to the combinatorics of tableaux.  We also establish the one-pointed case of a conjecture of Chan and Pflueger, expressing $\chi(\OO_{G_d^{\ba}(C,P)})$ as an enumeration of {\em set-valued tableaux}.\footnote{Chan and Pflueger have now proved their conjecture in the general two-pointed case, expressing $\chi\left(\OO_{G_d^{\ba, \bb}(C,P,Q)}\right)$ as an enumeration of set-valued \textit{skew} tableaux in \cite{cp}. }

\medskip

To conclude this introduction, we briefly sketch the argument for the \textit{classical} case of our main theorem, describing the Euler characteristic of the locus $W^r_d(C) \subseteq \Pic^d(C)$.  The construction of $W^r_d(C)$ as a degeneracy locus is standard; see \cite{kl}, \cite[\S VII]{acgh}, or \cite[(14.4.5)]{fulton-it}.

Fix a point $P$ on a smooth curve $C$, and let $\cL$ be a Poincar\'e bundle on $C\times \Pic^d(C)$, normalized so that $\cL|_{\{P\}\times \Pic^d(C)}$ is trivial.  Choose a nonnegative integer $n$ large enough so that all divisors of degree $n+d$ are non-special; any nonnegative $n\geq 2g-1-d$ will do.  Writing $\pr$ and $\prr$ for the projections from $C\times \Pic^d(C)$ to $C$ and $\Pic^d(C)$, respectively, let $\cE = \prr_*(\cL\otimes \pr^*\OO_C(nP))$ and $\cF = \prr_*(\cL\otimes \pr^*\OO_{nP})$.  Then the exact sequence on $C$
\[
  0 \to \OO_C \to \OO_C(nP) \to \OO_{nP} \to 0
\]
transforms via $\prr_*(\cL\otimes \pr^*(\cdot))$ into an exact sequence
\[
  0 \to \prr_*\cL \to \cE \xrightarrow{\varphi} \cF
\]
on $\Pic^d(C)$.  The Brill-Noether variety $W^r_d(C)$ is thereby identified with the locus in $\Pic^d(C)$ where $\dim\ker(\varphi)\geq r+1$.

Since $L(nP)$ is non-special for all $L$ in $\Pic^d(C)$, Riemann-Roch shows that the sheaf $\cE$ is locally free of rank equal to $h^0(C,L(nP))$; that is,
\[
  \rk(\cE) = n+d-g+1.
\]
The sheaf $\cF$ is also locally free, of rank
\[
  \rk(\cF)  = n,
\]
and, in fact, $\cF$ has a filtration $\cF=\cF_n\twoheadrightarrow \cF_{n-1}\twoheadrightarrow \cdots \twoheadrightarrow \cF_{1}=\OO_C$ with $\textrm{Ker}(\cF_{i}\twoheadrightarrow \cF_{i-1})$ trivial for all $i$ --- to see this, apply $\prr_*(\cL\otimes \pr^*(\cdot))$ to the exact sequence
\[
  0 \to \OO_{P} \to \OO_{nP} \to \OO_{(n-1)P} \to 0
\]
and use that $\prr_*(\cL\otimes \pr^*(\OO_{P}))=\OO_C$, from the normalization of $\cL$.
This means the Chern classes of $\cF$ are trivial, so $c(\cF-\cE)=c(-\cE)$.

The Brill-Noether dimension estimate comes from a basic fact about matrices: the locus of $q\times p$ matrices having kernel of dimension at least $t$ has codimension $t(q-p+t)$ inside the affine space of all matrices.  (Take $t=r+1$, $p=n+d-g+1$, and $q=n$ to get the Brill-Noether number.)  Applying the K-theoretic Giambelli formula of \cite{a} yields
\[
  \left[\OO_{W^r_d(C)}\right] = \left| \sum_{k\geq 0} \binom{g-d+r+k-1}{k}(-1)^k c_{g-d+r+j-i+k}(-\cE) \right|_{1\leq i,j\leq r+1}
\]
in $K(\Pic^d(C))$, whenever $\dim W^r_d(C) = \rho(g,r,d)$.  The Euler characteristic formula is then deduced from Hirzebruch-Riemann-Roch and some linear algebra (see \S\ref{Exs}).

\bigskip
\noindent
{\it Acknowledgements.}  Our initial motivation for this project came from studying \cite{clpt} and \cite{himn}, and we thank these authors for their inspiring work.  This collaboration began at the Fields Institute Thematic Program on Combinatorial Algebraic Geometry, and we are grateful to the organizers and the Institute for providing a stimulating working environment. We would like to thank Melody Chan for pointing us to   \cite{lenart}, and Allen Knutson for discussion about intersecting Schubert varieties. For \S\ref{Schub} we borrow the latex code for displaying pipe dreams from \cite{kmy}, whose  original version is credited there to Nantel Bergeron. We are grateful to the anonymous referees for their careful reading and for numerous suggestions which have resulted in a much improved manuscript.

\section{Background and preliminaries}
\label{background}

We begin by reviewing some of the basic facts we will need in proving Theorem~\ref{t.degloci}.

\subsection{Connective K-theory}

Our main theorem about degeneracy loci gives formulas in the connective K-homology of an algebraic variety $X$.  Foundational facts about this theory can be found in \cite{cai,dl}, and briefer digests are in \cite{himn}, \cite{himn2},  and \cite[Appendix~A]{a}.  The main features we will require are the following:

\medskip

\begin{enumerate}[(a)]
\item The connective K-homology $CK_*(X)$ is a graded module over $\ZZ[\beta]$, with $\deg \beta=1$.

\medskip

\item There are Chern classes operators for vector bundles; for a vector bundle $E$ on $X$, if $\alpha\in CK_*(X)$, then $c_k(E)\cdot \alpha \in CK_{*-k}(X)$.

\medskip

\item Specializing $\beta=0$ and $\beta=-1$ induces natural isomorphisms
\[
  CK_*(X)/(\beta=0) \isom A_*(X) \quad \text{ and } \quad CK_*(X)/(\beta=-1) \isom K_\circ(X)
\]
with Chow homology and the Grothendieck group of coherent sheaves, respectively.

\medskip

\item There are fundamental classes $[Z]\in CK_*(X)$ for closed subvarieties $Z\subseteq X$, specializing to $[Z]\in A_*(X)$ and $\left[\OO_Z\right]\in K_\circ(X)$.
\end{enumerate}

\medskip
A degeneracy locus inherits its scheme structure by pullback from a universal degeneracy locus.  In exploiting this, the key statement we need is this (cf.~\cite[Lemma, p.~108]{fulton-pragacz} and \cite[Theorem~7.4]{dl}):

\begin{lemma}\label{l.pullback}
Let $f$ be a morphism from a pure-dimensional Cohen-Macaulay scheme $X$ to a nonsingular variety $Z$.  Suppose $Y \subseteq Z$ is a Cohen-Macaulay subscheme of pure codimension $d$.  Then $W=f^{-1}Y$ has codimension $\leq d$.  If $W$ has pure codimension $d$ in $X$, then it is Cohen-Macaulay and $[W]=f^*[Y]$ in $CK_*(X)$.
\end{lemma}

\begin{proof}
Everything except the last statement is contained in \cite[Lemma, p.~108]{fulton-pragacz}, and the equality $[W]=f^*[Y]$ is also proved there for cohomology (or Chow) classes.  So it suffices to prove this equality for K-theory, which we do by a slight refinement of the standard argument for cohomology.  Let $\Gamma_f\subseteq X\times Z$ be the graph of $f$, so $W$ is identified with $\Gamma_f \cap (X\times Y)$ via the first projection.

If $\dim Z=m$, then the graph $\Gamma_f \subseteq X\times Z$ is locally cut out by a regular sequence $z_1,\ldots,z_m$; that is, the Koszul complex $K_\bullet(z)$ is exact and resolves $\OO_{\Gamma_f}$.  Indeed, there is an exact sequence
\[
  0 \leftarrow \OO_{\Gamma_f} \leftarrow T \leftarrow \exterior^2 T \leftarrow \cdots \leftarrow \exterior^m T \leftarrow 0,
\]
where $T = pr_2^*\,T^\vee_{Z}$ is the cotangent bundle of $Z$, pulled back to $X\times Z$.

Since $X\times Y$ is Cohen-Macaulay and $W\isom \Gamma_f \cap (X\times Y)$ has codimension $d+m$ in $X\times Z$, the restrictions $\bar{z}_1,\ldots,\bar{z}_m$ to $X\times Y$ also form a regular sequence.  This means the Koszul complex $K_\bullet(\bar{z}) = K_\bullet(z)\otimes \OO_{X\times Y}$ is also exact, so by restricting the above resolution to $X$ via the graph morphism, we obtain an exact sequence
\[
    0 \leftarrow \OO_{W} \leftarrow T\otimes \OO_{X} \leftarrow \exterior^2 T \otimes \OO_{X}\leftarrow \cdots \leftarrow \exterior^m T\otimes \OO_{X} \leftarrow 0.
\]
Since $f^*[\OO_{Y}] =\sum_i (-1)^i\, [\mathrm{Tor}^Z_i(\OO_X,\OO_{Y})]$ by definition, we see that $f^*[\OO_{Y}]=\allowbreak [\OO_X \otimes_{\OO_{Z}} \OO_{Y}] \allowbreak=[\OO_W]$, since exactness of the above sequence shows that the higher Tor terms vanish.
\end{proof}

\begin{remark}\label{r.technical1}
In \cite{himn}, formulas are proved in connective K-cohomology $CK^*(X)$, under the hypothesis that $X$ is smooth.  The relationship with our more general setup is best described in the framework of the {\it operational cohomology theory} associated to a (generalized oriented Borel-Moore) homology theory \cite{ap,gk}.  One can define $CK^*$ to be the operational cohomology ring associated to the homology theory $CK_*$, so that $CK^*(X)$ is defined for any scheme.  This is a graded algebra over $\ZZ[\beta]$, where now $\beta$ has degree $-1$, and $CK_*(X)$ is a module for $CK^*(X)$, with $c\in CK^i(X)$ acting as a homomorphism $CK_*(X) \to CK_{*-i}(X)$.

Specializing at $\beta=0$ and $\beta=-1$ produces natural isomorphims
\[
  CK^*(X)/(\beta=0) \isom A^*(X) \quad \text{ and }\quad CK^*(X)/(\beta=-1) \isom \mathrm{op}K^\circ(X),
\]
where $A^*(X)$ is the Fulton-MacPherson operational Chow ring, and $\mathrm{op}K^\circ \!(X)$ is the operational K-theory developed in \cite{ap}.  When $X$ is smooth, Poincar\'e isomorphisms show that the operational $CK^*(X)$ agrees with the connective K-cohomology used in \cite{himn}, and that $CK^*(X) \isom CK_{\dim X - *}(X)$.
\end{remark}

\begin{remark}\label{r.technical2}
Higher connective K-groups are defined and studied in \cite{cai} and \cite{dl}.  The two versions coincide in the part corresponding to the Grothendieck group $K_\circ$, but they diverge in general.  Both work in the category of quasi-projective schemes.  For an explanation of how to extend results to general schemes, see \cite{descent}.  In particular, one can construct Chern classes using the projective bundle formula and Grothendieck's method, as in \cite{cai} (and \cite{a}).
\end{remark}

\subsection{The degeneracy loci \texorpdfstring{$W_{\bp,\bq}$}{W} and \texorpdfstring{$\Omega_{\bp,\bq}$}{Omega}}
\label{ss.321}

\medskip
Now we turn to the degeneracy locus setup.  We have a sequence of vector bundles
\[
E_{p_t}\hookrightarrow \cdots \hookrightarrow E_{p_1} = E \xrightarrow{\varphi} F = F_{q_1} \twoheadrightarrow \cdots \twoheadrightarrow F_{q_t}
\]
on a (now possibly singular) variety $X$, where subscripts indicate rank, so that
\[
  0< p_t<\cdots<p_1 \quad \text{ and } \quad q_1>\cdots>q_t>0.
\]
It will be convenient to assume that the flag $E_{p_t}\hookrightarrow \cdots \hookrightarrow E_{p_1}$ extends to a \textit{full} flag $E_{1}\hookrightarrow \cdots \hookrightarrow E_{p_1}$ of sub-bundles of $E$ defined on $X$, and similarly, the flag $F_{q_1} \twoheadrightarrow \cdots \twoheadrightarrow F_{q_t}$ extends to a \textit{full} flag $F_{q_1} \twoheadrightarrow \cdots \twoheadrightarrow F_{1}$ of quotients of $F$ defined on $X$. This is harmless, as there exists always a full flag extending a given  partial flag, possibly after replacing $X$ with $X'$ such that $X'\rightarrow X$ is a tower of projective bundles, so that $CK_*(X)\hookrightarrow CK_*(X')$.

Let $V:=E\oplus F$.  The vector bundle $V$ includes isomorphic copies of the sub-bundles $E_\bullet$ via the graph $E_{\varphi}$ of~$\varphi$:
\[
  E_{p_t} \subset \cdots \subset E_{p_1} = E_{\varphi} \subseteq V,
\]
and it also comes with natural projections $V \twoheadrightarrow F_{q_i}$ for all $i$.

Our degeneracy loci lie in $X$, and in a Grassmann bundle over $X$:
\begin{equation}\label{e.deg-diagram}
\begin{tikzcd}
  \Omega_{\bp,\bq} \arrow[hookrightarrow]{r}  \arrow[rightarrow]{d} & \bGr(t,V) \arrow[rightarrow]{d}{\pi} \\
   W_{\bp,\bq}   \arrow[hookrightarrow]{r} &  X.
\end{tikzcd}
\end{equation}
The locus $W_{\bp,\bq}\subseteq X$ is defined by the conditions
\[
  \dim\ker( E_{p_{j}} \to F_{q_i} )\geq 1+i-j \qquad\mbox{for all $i,j$,}
\]
where here we usually assume
\renewcommand{\theequation}{$*$}
\begin{equation}\label{e.skew}
  q_i \geq p_{i} -1   \quad\text{ for all }i.
\end{equation}
(Evidently, it suffices to require these conditions only for $j\leq i$.  Later we will see that the ones for $j=i$ are enough.)

\renewcommand{\theequation}{\arabic{equation}}
\addtocounter{equation}{-1}

To define the locus $\Omega_{\bp,\bq}\subseteq \bGr(t,V)$, let $\bS\subseteq V$ be the tautological rank $t$ sub-bundle on $\bGr(t,V)$.  (Here $V$ should be understood as $\pi^*V$ --- following a common abuse, we omit notation for such pullbacks.)  Using the inclusions and projections $E_{p_j} \hookrightarrow V \twoheadrightarrow F_{q_i}$ described above, $\Omega_{\bp,\bq}\subseteq \bGr(t,V)$ is defined by the conditions
\begin{align}
\label{Sconditions}
  \dim( \bS \cap E_{p_j} ) \geq t+1-j \quad \text{ and }\quad  \dim\ker( \bS \to F_{q_i} ) \geq i \qquad\mbox{for all $1\leq i,j\leq t$.}
\end{align}
No restrictions on $\bp$ and $\bq$ are needed here.  Note that $\Omega_{\bp,\bq'} \subseteq \Omega_{\bp,\bq}$ if $q'_i\geq q_i$ for all $i$.  (And likewise, $\Omega_{\bp',\bq} \subseteq \Omega_{\bp,\bq}$ if $p'_j \leq p_j$ for all $j$.)  From the definition, the fiber of $\Omega_{\bp,\bq} \to X$ over any point $x\in X$ is an intersection of two Schubert varieties in the Grassmannian $\Gr(t,V|_x)$.

\subsection{321-avoiding permutations}
In analyzing the relationship between $\Omega_{\bp,\bq}$ and $W_{\bp,\bq}$, we will need some combinatorics of permutations and Schubert varieties.  For any permutation $v$, there is a rank function
\[
  r_v(a,b) := \#\{ i\leq a \,|\, v(i)>b \}.
\]
These define Schubert varieties in the flag variety (or a degeneracy locus on any variety with flagged vector bundles), by imposing conditions
\[
  \dim \ker ( E_a \to F_b ) \geq r_v(a,b) \qquad \mbox{for all $(a,b)$.}
\]
The Bruhat order on permutations describes containment of Schubert varieties; equivalently, given two permutations $u$ and $v$, one has $u\leq v$ if and only if $r_u(a,b) \leq r_v(a,b)$ for all $(a,b)$.

In fact, the above conditions are redundant, and one can find a much shorter list of conditions.  Suppose one has a collection of pairs $S=\{(a_i,b_i)\}_i$ and corresponding integers $k_i$ such that the set of permutations $u$ satisfying $r_u(a,b)\geq k_i$ has a unique minimum $v$ in Bruhat order.  Then the Schubert variety (or degeneracy locus) corresponding to $v$ is determined (scheme-theoretically) by the conditions 
\[
  \dim \ker (E_{a_i} \to F_{b_i}) \geq k_i \qquad \mbox{for $(a_i,b_i)$ in $S$.}
\]

One choice of $S$ is given by Fulton's essential set for the permutation $v$ \cite[Lemma 3.10]{fulton-flags}.  For the permutations arising in our situation, we will use a different choice.

For any $\bp,\bq$, we define an associated permutation $w$ by setting
\[
  w(p_i) := \max\{ q_{i}+1, \, p_i \}  \quad \text{ for } 1\leq i\leq t,
\]
and then filling in the remaining entries minimally with unused numbers in increasing order.  For example, if $\bp=(5,4,1)$ and $\bq=(5,2,1)$ then
\begin{equation}
\label{eq:wex}
  w = {\bf2}\;1\;3\;{\bf4}\;{\bf6}\;5.
\end{equation}
Given $\bp,\bq$, let us also define a sequence $\bq'$ by
\[
 q'_i = \max\{ q_i,\, p_i-1\}.
\]
The associated permutations for $\bp,\bq$ and  for $\bp,\bq'$ are the same, but note that $w(p_i) = q_i'+1$ for all $i$.  The new pair $\bp,\bq'$ satisfies \eqref{e.skew} by definition.  

In fact, the permutation $w$ is a {\it 321-avoiding permutation} (i.e., there are no $a<b<c$ such that $w(c)> w(b)>w(a)$), and all 321-avoiding permutations arise this way for some $\bp,\bq$, since any such permutation is a shuffle of two increasing subsequences (see e.g., \cite{el}).  More precisely, $w$ is obtained by shuffling $(q'_t+1,\dots, q'_1+1)$  with the sequence of left-over numbers.  We have $w(p_i) = q_i'+1 \geq p_i$ for all $i$, and $w(p)\leq p$ for all $p$ not among the $p_i$.  It follows that
\begin{equation}\label{e.w-ess}
 r_w(p_i,q'_i)=1 \quad \text{and} \quad r_w(p_{i+1},q'_i)=0,
\end{equation}
which will be useful later, in the proof of Proposition~\ref{p.fibers}.

From the construction, one sees that $w$ may also be characterized as the (unique) minimal element among all $u$ such that $r_u(p_i,q'_i) \geq 1$, for $\bp,\bq'$ as above.

\begin{lemma}\label{l.w}
The permutation $w$ is the unique minimal one in Bruhat order such that
\[
  \#\{p\leq p_{i}\,|\, w(p)>q_{i} \} \geq 1 \qquad\mbox{for all $i$.}
\]
Its length is equal to $\sum_{i=1}^t(q'_i-p_{i}+1)$.
\end{lemma}

The length of $w$ is defined to be $\#\{a<b\,|\,w(a)>w(b)\}$; it is the codimension of the corresponding Schubert variety in the flag variety.  The lemma implies that the conditions specified by $\bp,\bq$ are equivalent to those given by $\bp,\bq'$, that is, $W_{\bp, \bq}= W_{\bp, \bq'}$. Its proof is fairly straightforward, since the condition is trivial whenever $q_i\leq p_i-1$. 

It follows that among the conditions
\[
 \dim \ker( E_{p_j} \to F_{q_i} ) \geq 1+i-j
\]
defining $W_{\bp,\bq}$, those with $i=j$ are sufficient.

\subsection{Schubert varieties in flag bundles}

Next we consider a special case of the degeneracy locus problem (which turns out to be the universal situation).  We have a variety $Y$, with a vector bundle $V_Y$ of rank $p_1+q_1$ and quotient bundles $V_Y \twoheadrightarrow F_{q_i}$ of ranks $q_i$.  Let $X = \bFl(\bp,V_Y) \to Y$ be the flag bundle, and let $E_\bullet$ be the tautological flag of sub-bundles.  As usual, we suppress notation for pullbacks, writing $V=V_Y$ and $F_\bullet$ for the corresponding bundles on $X$.

In this setting, the diagram \eqref{e.deg-diagram} takes the form
\begin{equation}\label{e.deg-diagram2}
\begin{tikzcd}
  \bOmega_{\bp,\bq} \arrow[hookrightarrow]{r}  \arrow[rightarrow]{d} &  \bFl(\bp,V_Y) \times_Y \bGr(t,V_Y) \arrow[rightarrow]{d}{\hat\pi} \\
   \bW_{\bp,\bq}   \arrow[hookrightarrow]{r} &  X=\bFl(\bp,V_Y).
\end{tikzcd}
\end{equation}
The degeneracy locus $\bW_{\bp,\bq}$ is a Schubert variety in the flag bundle $X$, so it corresponds to a permutation; in fact, this is the permutation $w$ of Lemma~\ref{l.w}.

\begin{proposition}\label{p.fibers}
Assume the above situation, so $X=\bFl(\bp,V_Y)$.  Given $\bp,\bq$, let $\bq'$ be defined as before, that is, $q'_i=\max\{q_i,\,p_i-1\}$.  Then $\hat\pi(\bOmega_{\bp,\bq})=\bW_{\bp,\bq'}=\bW_{\bp,\bq}$.
\end{proposition}

\begin{proof}
The statement is local, so in proving it we may reduce to the case where $Y$ is a point. In this case, $X=\Fl(\bp,V)$ is a partial flag variety, and $\Omega_{\bp,\bq} \subseteq \Fl(\bp,V) \times \Gr(t,V)$.  Let $\pi\colon \Fl(\bp,V) \times \Gr(t,V) \to \Fl(\bp,V)$ and $\phi\colon \Fl(\bp,V) \times \Gr(t,V) \to \Gr(t,V)$ be the projections.

The conditions \eqref{Sconditions} defining $\Omega_{\bp,\bq}$ imply that after forgetting the $t$-dimensional subspace $S\subseteq V$, one has
\begin{equation}\label{forgetS}
 \dim \ker( E_{p_j} \to F_{q_i} ) \geq 1+i-j \quad \text{for all }i,j.
\end{equation}
By Lemma~\ref{l.w}, these conditions with $i=j$ imply the rank conditions given by the permutation $w$ associated to $\bp,\bq$, and thus
define the Schubert variety $W_{\bp,\bq'} \subseteq \Fl(\bp,V)$.  So it follows that $\pi(\Omega_{\bp,\bq}) \subseteq W_{\bp,\bq'}$.

On the other hand, the projection $\pi\colon \Omega_{\bp,\bq} \to W_{\bp,\bq'}$ is $B$-equivariant, for the standard action on $\Fl(\bp,V)\times \Gr(t,V)$ of a Borel subgroup $B\subseteq GL(V)$ fixing the flag $F_\bullet$.  To show that $\pi$ is surjective, it suffices to show that the fiber is nonempty over a general flag $A_\bullet$ in $W_{\bp,\bq'}$, i.e., a flag $A_\bullet$ such that
\[
  K_{i,i} := \ker( A_{p_i} \to F_{q'_i} ) \quad \mbox{satisfies $\dim K_{i,i}=1$ for all $i$,}
\]
 and
\[
  K_{i+1,i} := \ker( A_{p_{i+1}} \to F_{q'_i} ) = 0 \quad \mbox{for all $i$.}
\]
The dimensions  here follow from \eqref{e.w-ess}.  We see that the vector spaces $K_{i,i}$  give independent lines, since $K_{i+1,i} = 0$ means $K_{i,i} \cap A_{p_{i+1}}=0$.  So
\[
  S = K_{1,1} \oplus K_{2,2} \oplus \cdots \oplus K_{t,t} \subseteq A_{p_1}
\]
has dimension $t$, and $S\cap A_{p_j} = K_{j,j} \oplus \cdots \oplus K_{t,t}$ has dimension $t+1-j$.  Since there is a surjection $F_{q_i'} \twoheadrightarrow F_{q_i}$, one sees $\ker(S \to F_{q_i})$ contains $K_{1,1} \oplus \cdots \oplus K_{i,i}$, and therefore has dimension at least $i$.  We conclude that $\pi(\Omega_{\bp,\bq}) = W_{\bp,\bq'}$.
\end{proof}

Next we turn to the singularities and dimensions of our degeneracy loci.  At this point it will help to use partition notation.  As in the introduction, we define partitions $\lambda$ and $\mu$ by
\[
  \lambda_i = q_i-t+i \quad \text{ and } \quad \mu_j=p_{j}-(t+1-j).
\]
The condition \eqref{e.skew} is equivalent to requiring that $\lambda_i\geq \mu_i$ for all $i$, i.e., $\lambda/\mu$ is a skew shape.  We use the notation $\lambda'$ to denote the partition given by $\lambda'_i = q'_i-t+i$.  

\begin{proposition}\label{p.fibers2}
Assuming $X=\bFl(\bp,V_Y)$, the locus $\bOmega_{\bp,\bq}$ is reduced (or is Cohen-Macaulay, or has rational singularities) if $Y$ is reduced (resp., is Cohen-Macaulay, has rational singularities).  The same is true of $\bW_{\bp,\bq'}$.  

The dimensions of these loci are $\dim \bOmega_{\bp,\bq} = \dim X -|\lambda|+|\mu|$ and $\dim \bW_{\bp,\bq'} = \dim X - |\lambda'/\mu|$.
\end{proposition}

\begin{proof}
Again the statements are local on $Y$ (and preserved by products), so we will assume $Y$ is a point and show that the varieties in question have rational singularities (which implies Cohen-Macaulay and reduced).  Since $W_{\bp,\bq'}$ is a Schubert variety, it has rational singularities; its codimension is the length of $w$, which was calculated in Lemma~\ref{l.w} and is equal to $|\lambda'/\mu|$.  We focus on $\Omega_{\bp,\bq}$, using a description which will be useful later.

Recall that $\phi\colon \Fl(\bp,V) \times \Gr(t,V) \to \Gr(t,V)$ is the second projection.  Then
\[
  \Omega_{\bp,\bq} = \Omega' \cap \phi^{-1}\Omega_\lambda,
\]
where
\begin{align*}
  \Omega_\lambda &:= \{ S\,|\, \dim\ker( S \to F_{q_i} ) \geq i \text{ for all }i\} \subseteq \Gr(t,V), \quad \text{and} \\
  \Omega' &:= \{ (A_\bullet,S)\,|\,\dim( S \cap A_{p_j} ) \geq t+1-j \text{ for all }j\} \subseteq \Fl(\bp,V) \times \Gr(t,V).
\end{align*}
Restricting the projections $\pi$ and $\phi$ to $\Omega'$ produces flat morphisms (which we will denote by the same letter).  In fact, they are locally trivial fiber bundles, and we can describe their fibers explicitly.

The fiber of the first projection $\pi\colon \Omega' \to \Fl(\bp,V)$ over a flag $A_\bullet$ is a Schubert variety $\Omega_\nu(A_\bullet) \subseteq \Gr(t,V)$.  Here $\nu=\mu^\vee$ is the complementary partition to $\mu$ inside the $t\times (p_1+q_1-t)$ rectangle; specifically, $\nu_j = p_1+q_1-t-p_{t+1-j}+j$.  It follows that
\[
  \dim \Omega' = \dim \Fl(\bp,V) + |\mu|,
\]
which we will use again in the proof of Theorem~\ref{t.deg}.

The fiber of the second projection $\phi\colon \Omega' \to \Gr(t,V)$ is a Schubert variety in $\Fl(\bp,V)$.  Its corresponding permutation is the inverse of the Grassmannian permutation for the partition $\nu$.  In particular, intersecting with $\phi^{-1}\Omega_\lambda$, the morphism
\[
  \phi \colon \Omega_{\bp,\bq} \to \Omega_\lambda
\]
is again flat, and both the base and fibers have rational singularities.  By \cite[Th\'eor\`eme 5]{elkik}, we conclude that $\Omega_{\bp,\bq}$ has rational singularities.  The formula for $\dim \Omega_{\bp,\bq}$ also follows.
\end{proof}

We conclude this section with the following statement:

\begin{proposition}\label{p.bir}
Assume  $X=\bFl(\bp,V_Y)$.  Given $\bp,\bq$, let $\bq'$ be defined as  $q'_i=\max\{q_i,\,p_i-1\}$.  When restricted to $\bOmega_{\bp,\bq'}\subseteq \bOmega_{\bp,\bq}$, the map $\hat\pi\colon \bOmega_{\bp,\bq'} \to \bW_{\bp,\bq'}$ is birational.
\end{proposition}

\begin{proof}
As the statement is local on $Y$, we can reduce to the case where $Y$ is a point.
The argument in the proof of Proposition \ref{p.fibers} to construct an element $S$ in the fiber of $\pi\colon \Omega_{\bp,\bq} \to W_{\bp,\bq'}$ over a generic point $A_\bullet$ of $W_{\bp,\bq'}$
 shows that $\ker(S \to F_{q'_i}) = K_{1,1} \oplus \cdots \oplus K_{i,i}$, so $S\in \Omega_{\bp,\bq'}$ is uniquely determined.  It follows that $\pi\colon \Omega_{\bp,\bq'} \to W_{\bp,\bq'}$ is generically bijective.

Being a Schubert variety, $W_{\bp,\bq'} \subseteq \Fl(\bp,V)$ is the closure of a Schubert cell $W^\circ$, which in turn is a principal homogeneous space for a certain subgroup of $B$.  By Proposition~\ref{p.fibers2}, $\Omega_{\bp,\bq'}$ is reduced.  Since $\pi$ is $B$-equivariant, it follows that the restriction $\pi^{-1}(W^\circ) \to W^\circ$ is an isomorphism.
\end{proof}

\section{A determinantal formula in K-theory}
\label{determinantalF}

Now we can state and prove the main theorem.  Given doubly indexed series $c(i,j)=\sum_{m\geq 0} c_m(i,j)$ for $1\leq i,j\leq t$, we define the determinant
\begin{align}
\begin{split}
\label{eq:Delta}
  \Delta_{\lambda/\mu}(c;\beta) &= \left| (1-\beta T)^{-\lambda_i+\mu_j} c_{\lambda_i-\mu_j+j-i}(i,j) \right|_{1\leq i,j\leq t}\\
  &= \left| \sum_{k\geq 0} \binom{\lambda_i-\mu_j+k-1}{k}\beta^k c_{\lambda_i-\mu_j+j-i+k}(i,j) \right|_{1\leq i,j\leq t}
\end{split}
\end{align}
as in the introduction.  In the statement of the theorem, we return to the general setting, no longer requiring $X$ to be a flag bundle.

\begin{theorem}\label{t.deg}
Let $X$ be a variety with rational singularities, and let $c(i,j) = c^K(F_{q_i}-E_{p_{j}})$ be the K-theoretic Chern classes.
\medskip
\begin{enumerate}[(i)]
\item
\label{t.deg-part1} 
The locus $\Omega_{\bp,\bq}\subseteq \bGr(t,V)$ has $\codim \Omega_{\bp,\bq} \leq |\lambda/\mu|+t(p_1+q_1-t)$.  If equality holds, then $\Omega_{\bp,\bq}$ is Cohen-Macaulay, and
\[
  \pi_*[\Omega_{\bp,\bq}] = \Delta_{\lambda/\mu}(c;\beta) \cdot [X] \qquad\mbox{in $CK_*(X)$.}
\]

\item
\label{t.deg-part2} 
Assume $\lambda_i\geq \mu_i$ for all $i$.  Then $W_{\bp,\bq}\subseteq X$ has $\codim W_{\bp,\bq} \leq |\lambda/\mu|$.  If equality holds, then $W_{\bp,\bq}$ is Cohen-Macaulay, and
\[
  [W_{\bp,\bq}] = \Delta_{\lambda/\mu}(c;\beta) \cdot [X] \qquad\mbox{in $CK_*(X)$.}
\]
\end{enumerate}
\end{theorem}

\noindent
The statement in (\ref{t.deg-part2}) specializes to Theorem~\ref{t.degloci} from the introduction.

In the course of the proof, we require some formulas from \cite{a}.  The first is the determinantal formula for a Grassmannian degeneracy locus.  (This appeared originally in \cite{himn}, in a slightly different form.)  Let $\Omega_\lambda$ be a Grassmannian degeneracy locus, defined by conditions \mbox{$\dim \ker ( E_t \to F_{q_i} ) \geq i$} for all $i$, where $\lambda_i=q_i-t+i$ as above.  Then
\begin{equation}\label{e.a1}
  [\Omega_\lambda] =  \left| (1-\beta T)^{-\lambda_i} c_{\lambda_i+j-i}\left(F_{q_i}-E_t\right) \right|_{1\leq i,j\leq t}.
\end{equation}
Next, the formal determinantal identity used in proving the ``general case'' of \cite[Theorem~1]{a} shows that
\begin{equation}\label{e.a2}
 \left| (1-\beta T)^{-\lambda_i} c_{\lambda_i+j-i}\left(F_{q_i}-E_t\right) \right|_{1\leq i,j\leq t} = \left| (1-\beta T)^{-\lambda_i} c_{\lambda_i+j-i}\left(F_{q_i}-E_{t+1-j}\right) \right|_{1\leq i,j\leq t}.
\end{equation}
Finally, suppose we have a tower of projective bundles
\begin{equation}\label{e.tower}
\PP\left(E_{p_1}/\bS_{t-1}\right) \xrightarrow{\tilde\pi^{(t)} }\cdots \xrightarrow{\tilde\pi^{(3)}} \PP\left(E_{p_{t-1}}/\bS_1\right) \xrightarrow{\tilde\pi^{(2)}} \PP\left(E_{p_t}\right) \xrightarrow{\tilde\pi^{(1)}} X,
\end{equation}
where $\bS_{j+1}/\bS_{j}\subset E_{p_{t-j}}/\bS_{j}$ is the tautological line bundle on the projective bundle $\PP\left(E_{p_{t-j}}/\bS_{j}\right)$ (suppressing notation for pullbacks of bundles under the natural projections $\tilde\pi^{(i)}$, as usual).  Then
\begin{equation}\label{e.a3}
 \tilde\pi^{(j)}_* \left( (1-\beta T)^{-m} c_m\left( F - \bS_{t+1-j} \right) \right) 
  = (1-\beta T)^{-m+p_j-(t+1-j)} c_{m-p_j+t+1-j}\left( F - E_{p_j} \right)
\end{equation}
for any bundle $F$ on $X$ and all $m$.  (This is \cite[Eq.~(5)]{a}.)

In broad strokes, the idea of the proof is simple.  As before, we first treat the case where $X=\bFl(\bp,V)$ is a flag bundle.  There we use the description of $\Omega_{\bp,\bq}$ as an intersection $\Omega' \cap \phi^{-1}\Omega_\lambda$, together with a resolution of $\Omega'$ and the determinantal formula \eqref{e.a1} to produce the desired formula.  This case is universal, and we deduce the general case by pullback, using Lemma~\ref{l.pullback}.

Now we turn to the details.

\begin{proof}[Proof of Theorem~\ref{t.deg}]
As before, we first suppose $X=\bFl(\bp,V_Y) \to Y$ is a flag bundle over a variety $Y$ with rational singularities, so $\bGr(t,V) = \bFl(\bp,V_Y)\times_Y \bGr(t,V_Y)$.  In this case, we have already seen in Proposition~\ref{p.fibers2} that the degeneracy loci have the expected codimensions:
\[
\codim \Omega_{\bp,\bq} = |\lambda/\mu|+t(p_1+q_1-t)
\]
and (when the pair $\bp,\bq$ satisfies \eqref{e.skew})
\[
\codim W_{\bp,\bq} = |\lambda/\mu|
\]
in $\bGr(t,V)$ and $X$, respectively.  In Proposition~\ref{p.fibers2}, we also saw that $\Omega_{\bp,\bq}$ and $W_{\bp,\bq}$ have rational singularities, hence they are in particular Cohen-Macaulay.

Recall from the proof of Proposition~\ref{p.fibers2} that $\Omega_{\bp,\bq} = \Omega' \cap \phi^{-1}\Omega_\lambda$, and this intersection is proper: the codimensions of $\Omega'$ and 
$\phi^{-1}\Omega_\lambda$ add to that of $\Omega_{\bp,\bq}$.  In particular, $[\Omega_{\bp,\bq}]=[\Omega']\cdot \phi^*[\Omega_\lambda]$.

The locus $\Omega'$ admits a desingularization by the variety $\tilde\Omega'$ parametrizing flags of sub-bundles $S_1\subset S_2\subset \cdots \subset S_t$ such that $\mathrm{rank}(S_j)=j$ and $S_{t+1-j}\subseteq E_{p_j}$.  In the tower of projective bundles \eqref{e.tower} this is $\tilde\Omega' = \PP(E_{p_1}/\bS_{t-1})$.  The rank $t$ bundle $\bS_t\subseteq E_{p_1} \subset V$ on $\tilde\Omega'$ defines a map
\[
 f\colon \tilde\Omega' \to \Omega' \subseteq \bGr(t,V), \qquad \left( S_1\subset S_2\subset \cdots \subset S_t \right) \mapsto S_t
\]
which is a desingularization.  (This is one of the standard desingularizations of Grassmannian Schubert varieties, going back to Kempf and Laksov.)  Write $\pi'\colon \tilde\Omega' \to X$ for the composition $\pi'=\pi\circ f$.

Since $\Omega'$ has rational singularities, $f_*[\tilde\Omega'] = \left[\Omega'\right]$.  By the projection formula, we obtain
\[
  f_*f^* [\Omega_{\lambda}] = [\Omega']\cdot \phi^*[\Omega_{\lambda}] = [\Omega_{\bp,\bq}].
\]
Now we can compute the pushforward as
\[
  \pi_*[\Omega_{\bp,\bq}] = \pi_*f_*f^* \phi^*[\Omega_{\lambda}] = \pi'_*f^* \phi^*[\Omega_{\lambda}].
\]
The formula \eqref{e.a1} for $[\Omega_\lambda]$ is preserved under pullback, and we can rewrite it using the identity \eqref{e.a2}:
\begin{align*}
  f^*\phi^*[\Omega_{\lambda}] &= \left| (1-\beta T)^{-\lambda_i} c_{\lambda_i+j-i}\left(F_{q_i}-\bS\right) \right|_{1\leq i,j\leq t} \\
  &= \left| (1-\beta T)^{-\lambda_i} c_{\lambda_i+j-i}\left(F_{q_i}-\bS_{t+1-j}\right) \right|_{1\leq i,j\leq t}.
\end{align*}
Finally, applying \eqref{e.a3} to the entries of the determinant gives
\begin{align*}
  \pi_*[\Omega_{\bp,\bq}] &= \pi'_*\left(  \left| (1-\beta T)^{-\lambda_i} c_{\lambda_i+j-i}\left(F_{q_i}-\bS_{t+1-j}\right) \right|_{1\leq i,j\leq t} \right) \\ 
   &= \left| (1-\beta T)^{-\lambda_i+p_{j}-(t+1-j)} c_{\lambda_i-p_{j}+(t+1-j)+j-i}\left(F_{q_i}-E_{p_{j}}\right) \right|_{1\leq i,j\leq t} \\
   &= \left| (1-\beta T)^{-\lambda_i+\mu_j} c_{\lambda_i-\mu_j+j-i}\left(F_{q_i}-E_{p_{j}}\right) \right|_{1\leq i,j\leq t},
\end{align*}
so we have the asserted formula for this locus.  When $\bp,\bq$ satisfy \eqref{e.skew}, that is, $q_i\geq p_i-1$ for all~$i$, then by Proposition~\ref{p.bir} the map $\pi\colon \Omega_{\bp,\bq} \to W_{\bp,\bq}$ is birational.  Since both loci have rational singularities, it follows that $\pi_*[\Omega_{\bp,\bq}]=[W_{\bp,\bq}]$, and the theorem is proved in this case.

Now we turn to the general situation where $X$ is an arbitrary variety with rational singularities.  We have the vector bundle $V=E_{p_1}\oplus F_{q_1}$ on $X$, as usual, and we form the flag bundle $\bFl = \bFl(\bp,V)\to X$ and Grassmann bundle $\bGr=\bGr(t,V) \to X$.  On $\bFl$, we have the tautological flag
\[
\mathbb{U}_\bullet: \mathbb{U}_{p_t} \subset \cdots \subset \mathbb{U}_{p_1} \subset V
\]
and the flag $E_\bullet$ on $X$ determines a section $\sigma\colon X \to \bFl$ such that $\sigma^{*}(\mathbb{U}_\bullet) = E_\bullet$.  The universal loci
\[
  \bW_{\bp,\bq} \subseteq \bFl \quad \text{and}\quad \bOmega_{\bp,\bq} \subseteq \bFl\times_X \bGr
\]
are defined by the same conditions defining $W_{\bp,\bq}$ and $\Omega_{\bp,\bq}$, respectively, using $\mathbb{U}_\bullet$ in place of $E_\bullet$.  The situation and notation are summarized in Figure \ref{fig:WO}.

\begin{figure}[htb]
\begin{center}
\begin{tikzcd}[back line/.style={densely dotted}, row sep=1.9em, column sep=6em]
& \bOmega_{\bp,\bq}  \ar[hookrightarrow]{rr} \ar[back line]{dd} 
  & & \bFl \times_X \bGr \ar{dd}{\hat\pi}  \\
\Omega_{\bp,\bq} \ar[hookrightarrow]{ur}  \ar[hookrightarrow, crossing over]{rr} \ar{dd} 
  & & \bGr  \ar[hookrightarrow, bend left=0]{ur}[sloped]{\hat\sigma} \\
& \bW_{\bp,\bq} \ar[hookrightarrow, back line]{rr} 
  & & \bFl   \\
W_{\bp,\bq} \ar[hookrightarrow]{rr} \ar[hookrightarrow, back line]{ur}   & & X \ar[crossing over, leftarrow]{uu}[swap, near end]{\pi} \ar[hookrightarrow]{ur}[sloped]{\sigma}
\end{tikzcd}
\end{center}
 \caption{The universal loci $\bW_{\bp,\bq} $ and $\bOmega_{\bp,\bq}$.}
   \label{fig:WO}
\end{figure}

The just-proved case of the theorem applies to these universal loci; in particular, they have rational singularities (so they are Cohen-Macaulay) and we have the asserted formulas for $\hat\pi_*[\bOmega_{\bp,\bq}]$ and $[\bW_{\bp,\bq}]$.  By construction, we have
\[
  W_{\bp,\bq} = \sigma^{-1}(\bW_{\bp,\bq}) \quad \text{and} \quad \Omega_{\bp,\bq} = \hat\sigma^{-1}(\bOmega_{\bp,\bq}).
\]
Since the loci $\bOmega_{\bp,\bq}$ and $\bW_{\bp,\bq}$ are Cohen-Macaulay, we may apply Lemma~\ref{l.pullback} to deduce the general statement of the theorem.
\end{proof}

\begin{remark}\label{r.CM}
In fact, one can further relax the hypothesis on $X$, requiring only that it be Cohen-Macaulay. Instead of using the flag bundle $\bFl$ as in the proof of Theorem~\ref{t.deg},
 after possibly replacing $X$ by $X'$ for an affine bundle $X' \to X$, one can assume that the vector bundles $E_\bullet$ and $F_\bullet$ are pulled back from a product of \textit{flag varieties} 
\[
  Z = \Fl\left(\bp,\mathbb{C}^N\right) \times \Fl\left(\mathbb{C}^N,\bq\right)
\]
for some sufficiently large $N$; here $\Fl\left(\mathbb{C}^N,\bq\right)$ is the flag variety parametrizing \textit{quotients} of $\mathbb{C}^N$.  (See for example \cite{graham-diag}.)  The loci $W_{\bp,\bq}$ and $\Omega_{\bp,\bq}$ are then pulled back from $Z$ and a Grassmann bundle over $Z$, respectively. Since $Z$ is nonsingular, the formula on $Z$ is then given by Theorem~\ref{t.deg},  so we can deduce it on $X$ via pullback, using Lemma~\ref{l.pullback} as in the proof of Theorem~\ref{t.deg} (we require $X$ to be Cohen-Macaulay to apply Lemma~\ref{l.pullback}).
\end{remark}

In the rest of the paper, we will discuss some  applications of the degeneracy locus formula.  One of them is a direct generalization of Kleiman and Laksov's proof of the existence theorem \cite{kl}: by standard intersection theory, one can deduce a criterion for non-emptiness of a degeneracy locus from a formula for its class.

\begin{corollary}\label{c.nonempty}
Let $\bp,\bq$, and $\bq'$ be as above, so the pair $\bp,\bq'$ satisfies \eqref{e.skew}, and let $\lambda'/\mu$ be the skew diagram corresponding to $\bp,\bq'$.  If $\Delta_{\lambda'/\mu}(c;0)\cdot[X] \neq 0$ in $A_*(X)$, then $W_{\bp,\bq'}$ is nonempty, and so $\Omega_{\bp,\bq}$ is also nonempty.

The converse holds when $X$ is projective: If $\codim W_{\bp,\bq'}\geq |\lambda'/\mu|$ and $W_{\bp,\bq'}$ (or equivalently, $\Omega_{\bp,\bq}$) is nonempty, then $\codim W_{\bp,\bq'}=|\lambda'/\mu|$ and $\Delta_{\lambda'/\mu}(c;0)\cdot[X]$ is nonzero in $A_*(X)$.
\end{corollary}

\section{Varieties of linear series as degeneracy loci}
\label{GandW}

We apply here Theorem~\ref{t.deg} to the study of the Brill-Noether theory of  a  smooth algebraic curve $C$ of genus $g$. We start by describing the degeneracy locus structure of two-pointed Brill-Noether varieties. 
 Given a linear series $\ell=(L,V)$ in $G^r_d(C)$ and a point $P\in C$, the definition of the vanishing sequence
\[
\bm{a}^{\ell}(P) = \left( 0\leq a^\ell_0(P) < \dots < a^\ell_r(P)\leq d \right)
\]
from the introduction can be equivalently phrased by saying that $\bm{a}^{\ell}(P)$ is the maximal sequence verifying the condition
\[
\dim \left(  V\cap H^0\Big(C,L\big(-a^\ell_{r+1-i}(P)\cdot P\big)\Big) \right) \geq i \qquad \mbox{for $1\leq i\leq r+1$.}
\]
Fixing two points $P$ and $Q$ in $C$ and two sequences $\ba$ and $\bb$, 
the variety of linear series $G^{\ba,\bb}_d(C, P, Q)$ is therefore defined by the conditions 
\begin{align}
\label{sections}
\begin{split}
  \dim\left( V\cap H^0(C,L(-a_{r+1-i} P)\right) &\geq i  \\
 \mbox{and}\quad \dim\left( V\cap H^0(C,L(-b_{r+1-i} Q)\right) &\geq i
\end{split}
\qquad\mbox{for all $1\leq i\leq r+1$.}
\end{align}
 We will construct $G^{\ba, \bb}_d(C, P, Q)$ as a degeneracy locus of type $\Omega_{\bp,\bq}$ inside a certain Grassmann bundle $\pi\colon\bGr \to \Pic^d(C)$, with indices $\bp$ and $\bq$ determined  below.

The construction generalizes the description of $W^r_d(C)$ reviewed in the introduction.  As before, choose $n\geq 0$ large enough so that line bundles of degree $d+n-b_r$ are non-special, that is, $n\geq 2g-1-d+b_r$.  Fix a Poincar\'e line bundle $\cL$ on $C\times\Pic^d(C)$, normalized so that $\cL|_{\{P\}\times\Pic^d(C)}$ is trivial.  Let $\pi_1$ and $\pi_2$ be the projections from $C\times \Pic^d(C)$ to $C$ and $\Pic^d(C)$, and set
\begin{align*}
  \cE_{j} &:= (\pi_2)_*\left(\cL\otimes \pi_1^*\OO_C(nP-b_{j-1}Q)\right) \quad \text{ and } \\
   \cF_{i} &:= (\pi_2)_*\left(\cL\otimes\pi_1^*\OO_{(n+a_{r+1-i})P}\right)
\end{align*}
for $1\leq i,j\leq r+1$.  The sheaf $\cE_j$ is a vector bundle of rank
\[
  p_j := \rk(\cE_j) = n+d-b_{j-1}+1-g,
\]
and $\cF_i$ is a vector bundle of rank
\[
  q_i:= \rk(\cF_i) = n+a_{r+1-i}.
\]
One thus obtains a sequence
\[
  \cE_{r+1} \hookrightarrow \cE_{r} \hookrightarrow \cdots \hookrightarrow \cE_{1}=: \cE \xrightarrow{\varphi} \cF :=\cF_{1} \twoheadrightarrow \cF_{2} \twoheadrightarrow \cdots \twoheadrightarrow \cF_{r+1}
\]
of vector bundles over ${\rm Pic}^d(C)$. 

Define $\mathcal{V}:=\cE \oplus \cF$, with  natural maps $\cV \twoheadrightarrow \cF_i$, and with $\cE_j \hookrightarrow \cV$ included via the graph of $\varphi$.  
Consider the Grassmann bundle 
\[
\pi\colon\bGr\left(r+1, \cV \right)\rightarrow {\rm Pic}^d(C), 
\]
and let $\mathbb{S}$ be the tautological rank $r+1$ sub-bundle on $\bGr\left(r+1, \cV\right)$.  
The locus in $\bGr\left(r+1, \cV \right)$ defined by the conditions 
\begin{align*}
\dim(\bS\cap \cE_1)\geq r+1  \quad \mbox{and} \quad
    \dim \ker(\bS\rightarrow \cF_i) &\geq i  \quad\text{ for all $i$,} 
\end{align*}
coincides with the locus of linear series $(L,V)\in G^r_d(C)$ such that
\begin{align*}
  \dim\left( V\cap H^0(C,L(-a_{r+1-i} P-b_0 Q)\right) &\geq i  \quad\text{ for all $i$.}
\end{align*}
Imposing the additional conditions 
\[
\dim(\bS\cap \cE_j)\geq r+2-j  \quad\mbox{for all $j$,} 
\]
as in \eqref{Sconditions},
one obtains the locus of linear series also satisfying
\[
  \dim\left( V\cap H^0(C,L(-a_{0} P-b_{j-1} Q)\right) \geq r+2-j  \quad\text{ for all $j$,}
\] 
hence satisfying (\ref{sections}).  Thus $G^{\ba, \bb}_d(C,P,Q)$ can be identified with the degeneracy locus $\Omega_{\bp,\bq} \subseteq \bGr\left(r+1, \mathcal{V} \right)$.

We now study the image of $G^{\ba,\bb}_d(C,P,Q)$ in ${\rm Pic}^d(C)$ via the map $\pi$.  Let $W^{\ba,\bb}_d(C,P,Q)$ be the degeneracy locus $W_{\bp,\bq}$ in $\Pic^d(C)$ as in \S \ref{determinantalF}, that is,
\[
W^{\ba,\bb}_d(C,P,Q) := \left\{ L\in {\rm Pic}^{d}(C) \, | \, \dim \ker(\cE_{j} \rightarrow \cF_i)\geq 1+i-j
\right\}.
\]
Equivalently, $W^{\ba,\bb}_d(C,P,Q)$ is the locus of line bundles $L \in {\rm Pic}^{d}(C)$ such that 
\[
h^0\left(C,L(-a_{r+1-i} P-b_{j-1}Q)\right)\geq 1+i-j \qquad\mbox{for all $i,j$.}
\]
Recall the definition of the two partitions $\lambda$ and $\mu$ associated to the data $g,d,\ba,\bb$:
\begin{align*}
\begin{split}
 \lambda_i &:= n+ a_{r+1-i}-(r+1-i) \\
 \mu_i &:= n-b_{i-1}+i-1 -g+d-r
\end{split}
\qquad\mbox{for $1\leq i\leq r+1$.}
\end{align*}
From Proposition \ref{p.bir}, when $\lambda/\mu$ is a skew shape, one has $\pi(G^{\ba,\bb}_d(C,P,Q))=W^{\ba,\bb}_d(C,P,Q)$, and in this case $\pi \colon G^{\ba,\bb}_d(C,P,Q)\rightarrow W^{\ba,\bb}_d(C,P,Q)$ is birational.  In general, let $\ba'$ be the sequence defined as 
\[
a'_i := a_i+ \max\{0,  d -g -a_i -b_{r-i}\} \qquad\mbox{for all $i$.}
\]
The diagram $\lambda'/\mu$ is a skew shape by construction, where 
\[
\lambda'_i=n+ a'_{r+1-i}-(r+1-i) \qquad\mbox{for all $i$.}
\]
 If $\lambda/\mu$ is already a skew shape, then $\ba'=\ba$ and $\lambda'=\lambda$.
We have the following diagram
\[
\begin{tikzcd}
G^{\ba, \bb}_d(C,P,Q) \arrow[hookrightarrow]{r}  \arrow[rightarrow]{d} & \bGr\left(r+1, \mathcal{V} \right) \arrow[rightarrow]{d}{\pi}\\
W^{\ba',\bb}_d(C,P,Q) \arrow[hookrightarrow]{r}  & \Pic^d (C)
\end{tikzcd}
\]
fitting into the framework studied in \S \ref{determinantalF}, so we can apply Theorem \ref{t.deg}.  

Since the vector bundles $\mathcal{F}_i$ have trivial Chern classes and therefore the K-theoretic Chern classes $c(i,j)=c^K(\mathcal{F}_i -\mathcal{E}_{j})$ are equal to $c^K(\mathcal{-E}_{j})$, the determinantal formula in \eqref{eq:Delta} gives
\begin{align}
\label{DeltaBNinCK}
\begin{split}
\Delta_{\lambda/\mu}(c;\beta) &=
\left|  (1-\beta T)^{-\lambda_i+\mu_j} c^K_{\lambda_i-\mu_j+j-i} \left(\mathcal{-E}_{j}\right) \right|_{1\leq i,j\leq r+1}  \\
&= \left|  \sum_{k\geq 0} {\lambda_i-\mu_j +k-1 \choose k} \beta^k c^K_{\lambda_i-\mu_j+j-i+k} \left(\mathcal{-E}_{j}\right) \right|_{1\leq i,j\leq r+1}.
\end{split}
\end{align}
Recall that the expected dimension of the pointed Brill-Noether locus $G^{\ba, \bb}_d(C,P,Q)$ is
\begin{align*}
 \rho(g,r,d,\bm{a},\bm{b}) = g - \sum_{i=0}^{r} (g-d + a_{i}+b_{r-i}) = g - |\lambda/\mu|.
\end{align*}
Assume that $G^{\ba, \bb}_d(C,P,Q)$ has dimension equal to $\rho$.  Then by Theorem \ref{t.deg} it is Cohen-Macaulay~and 
\begin{align}
\label{BNinCK}
\pi_* \left[{G^{\ba,\bb}_d(C,P,Q)} \right] = \Delta_{\lambda/\mu}(c;\beta) \qquad \mbox{in $CK^*\left(\Pic^{d}(C)\right)$.} 
\end{align}

Similarly, assume $\lambda/\mu$ is a skew shape.  If ${W^{\ba,\bb}_d(C,P,Q)}$ has dimension equal to $\rho$, then it is Cohen-Macaulay, and has class given by \eqref{DeltaBNinCK} in $CK^*\left(\Pic^{d}(C)\right)$.

The determinant in \eqref{DeltaBNinCK} will be further manipulated in \S\ref{Euler}. We conclude this section with the following statement:

\begin{proposition}
\label{prop:aa'GW}
Let $(C,P,Q)$ be any smooth two-pointed curve of genus $g$.  
If both $G^{\ba, \bb}_d(C,P,Q)$ and $W^{\ba',\bb}_d(C,P,Q)$ have the expected codimension, then they are Cohen-Macaulay and
\begin{align*}
\pi_* \left[{G^{\ba,\bb}_d(C,P,Q)} \right] &= \left(-\beta\right)^{|\bm{a}'|-|\bm{a}|} \Delta_{\lambda'/\mu}(c;\beta)\\
	&= \left(-\beta\right)^{|\bm{a}'|-|\bm{a}|} \left[{W^{\ba',\bb}_d(C,P,Q)} \right] \qquad\qquad \mbox{in $CK^*\left(\Pic^{d}(C)\right)$.}
\end{align*}
\end{proposition}

\begin{proof}
For the first equality, we start by observing that for partitions $\lambda$ and $\mu$, the determinant $|c^K_{\lambda_i-\mu_j+j-i}|$ vanishes, unless $\lambda/\mu$ is a skew shape.  Indeed, if $\lambda_k<\mu_k$ for some $k$, then the matrix is singular, since it has $0$ in position $(i,j)$, for all $i\geq k \geq j$. 
Using \eqref{BNinCK}, the entries of the determinants $\Delta_{\lambda/\mu}(c;\beta)$ and $\Delta_{\lambda'/\mu}(c;\beta)$ from \eqref{DeltaBNinCK}
differ in those bottom rows $i$ where $\lambda_i<\mu_i$. For  $\Delta_{\lambda'/\mu}(c;\beta)$, such rows have entries equal to $1$ in position $(i,i)$, and $0$ in position $(i,j)$, for $j<i$.
For $\Delta_{\lambda/\mu}(c;\beta)$, expanding  by linearity in these bottom rows, one has vanishing contributions by the above observation, unless for each such row $i$, one considers the term with $k=-\lambda_i+\mu_i$. It follows that the nonzero contribution to the left-hand side is given by the determinant with entries equal to $(-\beta)^{-\lambda_i+\mu_i}$ in position $(i,i)$, and $0$ in position $(i,j)$, for $j<i$. The first equality follows since $-\lambda_i+\mu_i = a'_i-a_i$, for each $i$ such that $\lambda_i<\mu_i$.
The second equality follows from Theorem \ref{t.deg} and the fact that $\lambda'/\mu$ is a skew shape.
\end{proof}

\begin{remark}
When $\bm{b}=(0,\dots,r)$, the two-pointed Brill-Noether varieties  specialize to the one-pointed Brill-Noether varieties $G^{\ba}_d(C,P)$ and $W^{\ba}_d(C,P)$ which parametrize respectively linear series and line bundles on $C$ with imposed vanishing sequence $\ba$ at a single point $P$. In this case, the role of the flag of sub-bundles of $\mathcal{E}$ is redundant in the above construction, as it is enough to consider maps from $\mathcal{E}$ to the flag of quotients of $\mathcal{F}$.
With this description in place, the degeneracy locus formula from \cite{himn} suffices to compute the K-theory class  $\pi_*\left[{G^{\ba}_d(C,P)}\right]$. 
Explicit computations in the one-pointed case will be given in \S\S\ref{sec:1ptcase} and \ref{svt-one-pointed}.
\end{remark}

\section{Euler characteristics}
\label{Euler}

Our next goal is to give a formula for the Euler characteristic of the two-pointed Brill-Noether loci $G^{\ba,\bb}_d(C,P,Q)$.  In order to simplify the determinantal  formula (\ref{DeltaBNinCK}) for the K-classes of varieties of linear
series, we  prove some general lemmas on K-theoretic Chern classes and apply them to the bundles $\mathcal{E}_i$ in  \S \ref{GandW}.

Throughout this section, we specialize at $\beta=-1$, and take all Chow and numerical groups with coefficients in $\mathbb{Q}$.

\begin{lemma}\label{l.ch}
Suppose a rank-$r$ vector bundle $E$ has $\ch(E)_i=0$ for $i>1$.  Then $\ch\left(c^K_i(E)\right)= c_i(E)$. 
\end{lemma}

That is, if the Chern character of $E$ is $\ch(E) = r + c_1(E)$, then K-theory Chern classes agree with cohomology Chern classes under the Chern character isomorphism.

\begin{proof}
First recall that the Chern class of a line bundle has Chern character $\ch\left(c^K_1(L)\right) = 1 - \ee^{-c_1(L)}$, where $\ee^x$ is the formal power series $\sum_{k\geq 0} \frac{x^k}{k!}$.  Now let $L_1,\ldots,L_r$ be the K-theoretic Chern roots of~$E$, i.e., line bundle classes so that $c^K(E) = c^K(L_1)\cdots c^K(L_r)$.   Then
\begin{align*}
  \ch\left(c^K_i(E)\right) &= \ch \left( e_i \left( c^K_1(L_1),\ldots,c^K_1(L_r) \right) \right) \\
                &= e_i \left( \ch \left( c^K_1(L_1) \right), \ldots, \ch \left( c^K_1(L_r) \right) \right) \\
                &= e_i\left( 1-\ee^{-c_1(L_1)},\ldots, 1-\ee^{-c_1(L_r)} \right),
\end{align*}
where $e_i(x_1,\ldots,x_r)$ is the $i$-th elementary symmetric polynomial.  The lowest-degree term in the last line is equal to $c_i(E)$.  The higher-degree terms vanish, thanks to the hypothesis $\ch(E)_i=0$, for $i>1$, and Lemma~\ref{l.powers}.
\end{proof}

Given formal variables $x_1,\dots,x_n$, let $e_i(x_1,\dots,x_n)$  be the $i$-th elementary symmetric polynomial,  and  let $p_i (x_1,\dots,x_n):= x_1^i + \cdots +x_n^i$  be the $i$-th power sum symmetric polynomial.

\begin{lemma}
\label{l.powers}
Consider the ideal $I:=(p_2,\ldots,p_n) + (x_1,\ldots,x_n)^{n+1}$ in  $\mathbb{Q}\llbracket x_1,\dots,x_n\rrbracket$.  For $1\leq i\leq n$, one has
\[
  e_i\left(1-\ee^{-x_1},\ldots,1-\ee^{-x_n}\right) \equiv e_i\left(x_1,\ldots,x_n\right) \qquad\mbox{modulo $I$.}
\]
\end{lemma}

\begin{proof}
Write $e_i:=e_i(x_1,\dots,x_n)$ and $\overline{e}_i:=e_i\left(1-\ee^{-x_1},\ldots,1-\ee^{-x_n}\right)$, and similarly, write $p_i:=p_i(x_1,\dots,x_n)$ and $\overline{p}_i:=p_i\left(1-\ee^{-x_1},\ldots,1-\ee^{-x_n}\right)$. 
Since  
\[
p_1\left(1-\ee^{-x_1},\ldots,1-\ee^{-x_n}\right)= \sum_{k\geq 1} (-1)^{k-1}\frac{p_k}{k!}
\]
and $p_k\in (x_1,\ldots,x_n)^{n+1}$ for $k>n$, we have $\overline{p}_1 \equiv p_1$ modulo $I$.
It follows from Newton's identities that 
\begin{align*}
\overline{e}_i \equiv \frac{\overline{p}_{1}^i}{i!}  \mod (\overline{p}_2,\dots,\overline{p}_n) \qquad\mbox{and}\qquad
{e_i} \equiv \frac{{p}_{1}^i}{i!}  \mod ({p}_2,\dots,{p}_n)
\end{align*}
for  $1\leq i\leq n$.
Since $\overline{p}_2,\dots,\overline{p}_n \in I$, we conclude that 
$\overline{e}_i\equiv  \dfrac{\overline{p}_{1}^i}{i!} \equiv \dfrac{{p}_{1}^i}{i!}  \equiv e_i \mod I$,
for each $i$.
\end{proof}

Let $(C,P,Q)$ be a smooth two-pointed curve of genus $g$, and consider the vector bundles $\cE_i$ from \S\ref{GandW}.  Lemma~\ref{l.ch} applies to these bundles.  Indeed, modulo numerical (or homological) equivalence, 
the Chern classes of $-\cE_i$ are
\begin{align*}
 c_j\left(-\cE_i\right) = \frac{\theta^j}{j!},
\end{align*}
where $\theta$ is the cohomology class of the theta divisor.  (The proof given in \cite[\S VII]{acgh} is for singular cohomology, but it works as well in numerical or homological equivalence.)  
Equivalently, $\ch\left(-\cE_i\right)={\rm rank}\left(-\cE_i\right)+\theta$.  We therefore have
\begin{align}
\label{e.chE}
  \ch\Big( c^K_j\left(-\cE_i\right) \Big) = \frac{\theta^j}{j!}.
\end{align}

Now we can compute the Euler characteristic of the loci $G^{\ba,\bb}_d(C,P,Q)$ via Hirzebruch-Riemann-Roch.  The Todd class of $\Pic^d(C)$ is trivial, so
\begin{align*}
  \chi\left(\OO_{G^{\ba,\bb}_d(C,P,Q)}\right) &= \int_{\Pic^d(C)}  \ch \left( \pi_*\left[\OO_{G^{\ba,\bb}_d(C,P,Q)}\right] \right). 
\end{align*}  
Combining \eqref{BNinCK} and \eqref{e.chE} with the specialization of \eqref{DeltaBNinCK} at $\beta=-1$, the Euler characteristic is 
\begin{align}
\label{chiab}
\chi \left(\OO_{G^{\ba,\bb}_d(C,P,Q)} \right)= \int_{\Pic^d(C)}  \left| (1+T)^{-g+d-a_{r+1-i}-b_{j-1}+j-i} c_{g-d+a_{r+1-i}+b_{j-1}} \right|_{1\leq i,j\leq r+1}. 
\end{align}
From the Poincar\'e formula $\int \theta^g = g!$, it follows that the Euler characteristic is $g!$ times the coefficient of $\theta^g$ in the expansion of the determinant.  The next step is to analyze this expansion.

Let $\rho:=\rho(g,r,d,\ba,\bb)$, and recall that 
\[
\lambda_i-\mu_j +j-i= g-d+a_{r+1-i}+b_{j-1}.
\]
If we expand the operators $(1+T)^{-\lambda_i+\mu_j}$ in powers of $T$, the constant term is the cohomology class 
\begin{align}
\label{Wwhenkis0}
\left|  c_{\lambda_i-\mu_j +j-i}\right|_{1\leq i,j\leq r+1}, 
\end{align}
and is a multiple of $\theta^{g-\rho}$ (possibly zero).  The determinant in \eqref{chiab} is obtained by applying the operator $(1+T)^{\mu_j}$ to the $j$-th column of the  matrix in \eqref{Wwhenkis0}, and the operator $(1+T)^{-\lambda_i}$ to its $i$-th row. With binomial coefficients for negative integers $-s$  given by $ \binom{-s}{k}=\frac{-s(-s-1)\cdots(-s-k+1)}{k!}$, for $k\geq 0$, we have by linearity:
\begin{align*}
\sum_{|l/m|=|\lambda/\mu|+\rho } \left( \prod_{i=1}^{r+1} \binom{\mu_i}{\mu_i-m_i}\binom{-\lambda_i}{l_i-\lambda_i} \right)  \left| T^{l_i-\lambda_{i} + \mu_{j}-m_j}c_{\lambda_i-\mu_j +j-i} \right|_{1\leq i,j\leq r+1},
\end{align*}
the sum being over all sequences $l$ and $m$ with $l_i\geq \lambda_i$ and $m_i\leq \mu_i$.  (Here $l$ and $m$ are not required to be partitions, but we still use the notation $|l/m| = \sum (l_i-m_i)$.)

This proves Theorem~\ref{t.BN}.  More precisely, we have proved:

\begin{theorem}
\label{t.euler-two-pointed}
Let $(C,P,Q)$ be any smooth two-pointed curve of genus $g$.  
If $G^{\ba,\bb}_d(C,P,Q)$ has dimension equal to $\rho$, then the Euler characteristic
$\chi\left(\OO_{G^{\bm{a},\bm{b}}_d(C,P,Q)}\right)$ equals
\begin{align*}
\sum_{|l/m| =|\lambda/\mu|+\rho } g! \left( \prod_{i=1}^{r+1} \binom{\mu_i}{\mu_i-m_i}\binom{-\lambda_i}{l_i-\lambda_i} \right)  \left| \frac{1}{(l_i-m_j +j-i)!} \right|_{1\leq i,j\leq r+1}.
\end{align*}
Assume furthermore that $\lambda_i\geq \mu_i$ for all $i$.  If $W^{\bm{a},\bm{b}}_d(C,P,Q)$ has dimension equal to $\rho$, then 
\[
\chi\left(\OO_{W^{\bm{a},\bm{b}}_d(C,P,Q)}\right) = \chi\left(\OO_{G^{\bm{a},\bm{b}}_d(C,P,Q)}\right).
\]
\end{theorem}

\section{Determinantal and tableau formulas}
\label{Exs}

In this section, we will give a simplified expression for the Euler characteristic of the loci $G^{\bm{a},\bm{b}}_d(C,P,Q)$, expressing it as a weighted enumeration of standard Young tableaux, by performing a combinatorial analysis of the sum.  Along the way, we find a nonemptiness criterion for these loci, stated in Proposition~\ref{non-emptiness}.  Then we examine several special cases of particular interest.

We use the convention that a partition $\lambda$ corresponds to the shape with $\lambda_i$ boxes in the $i$-th row, where rows are indexed from top to bottom. There is a containment of shapes $\mu\subseteq \lambda$ when  two partitions $\lambda$ and $\mu$ satisfy $\lambda_i\geq \mu_i$ for all $i$. The {\it skew Young diagram}  $\lambda/\mu$ is represented as the complement of $\mu$ in~$\lambda$. 
 A \emph{standard Young tableau} on a skew shape $\lambda/\mu$ is a filling of the boxes of $\lambda/\mu$ by numbers $1,\ldots,|\lambda/\mu|$ such that the entries in each row and in each column are strictly increasing.  The number of standard Young tableaux  on $\lambda/\mu$ is commonly denoted by $f^{\lambda/\mu}$, and is given by the determinantal formula 
\begin{equation}
\label{e.standard-tableaux}
f^{\lambda/\mu} = |\lambda/\mu|! \left| \frac{1}{(\lambda_i-\mu_j+j-i)! }\right|_{1\leq i,j\leq r+1}
\end{equation}
(see \cite{aitken}).  
For example, for $\lambda=(3,1), \mu=(1,0)$, one has
\[
\lambda/\mu=
\ytableausetup{mathmode, boxsize=1.5em}
\begin{ytableau}
\none &  & \\
& \none & \none
\end{ytableau}
\qquad\mbox{and}\qquad f^{\lambda/\mu}= 3! \left| \begin{array}{cc} \dfrac{1}{2!} & \dfrac{1}{4!} \\[9pt]
0 & 1 \end{array} \right|=3.
\]

We extend the above notation to arbitrary sequences $l=(l_1,\ldots,l_{r+1})$ and $m=(m_1,\ldots, m_{r+1})$ of nonnegative integers, writing $l/m$ for a ``generalized skew diagram'' --- note that  we allow the differences $l_i-m_i$ to be negative.  Extending the notation for skew shapes, we will write
\[
  |l/m| := \sum_{i=1}^{r+1} (l_i-m_i)
\] 
and
\begin{equation}
\label{e.standard-tableaux-phi}
  f^{l/m} := |l/m|! \left| \frac{1}{(l_i-m_j+j-i)!}\right|_{1\leq i,j\leq r+1}.
\end{equation}

There are two basic facts underpinning our arguments in this section:

\medskip

\begin{enumerate}
\item[{\bf Fact 1.}] Suppose $\lambda$ and $\mu$ are {\em partitions} of length $r+1$.  Then
\[
  f^{\lambda/\mu} = |\lambda/\mu|! \left| \frac{1}{(\lambda_i-\mu_j+j-i)! }\right|_{1\leq i,j\leq r+1}
\]
is nonzero if and only if $\lambda_i\geq \mu_i$ for all $i$.  (Here one should read reciprocals of factorials of negative integers as $0$.)

\medskip

\item[{\bf Fact 2.}] Suppose $\lambda=(\lambda_1,\ldots,\lambda_{r+1})$ is a partition, and $l=(l_1,\ldots,l_{r+1})$ is any sequence of nonnegative integers such that $l_i\geq \lambda_i$ for all $i$.  If the sequence $(l_1-1,\ldots,l_{r+1}-(r+1))$ consists of distinct integers, and $w$ is the permutation which sorts them into decreasing order, then the sequence $\lambda^+_i = l_{w(i)}-w(i)+i$ is a partition with $\lambda^+_i\geq \lambda_i$ for all $i$.
\end{enumerate}

\medskip

\begin{proof}[Proof of Theorem~\ref{t.euler-two-pointed-tableau}]
We can rewrite the formula of Theorem~\ref{t.euler-two-pointed} as
\begin{align}
\chi\left(\OO_{G^{\bm{a},\bm{b}}_d(C,P,Q)}\right) &= \sum_{|l/m|=|\lambda/\mu| +\rho } 
\left( \prod_{i=1}^{r+1} \binom{\mu_i}{\mu_i-m_i}\binom{-\lambda_i}{l_i-\lambda_i} \right)  f^{l/m} \nonumber \\
 &= \sum_{|l/m|=|\lambda/\mu| +\rho } 
\left( \prod_{i=1}^{r+1} \binom{\mu_i}{\mu_i-m_i}\binom{l_i-1}{l_i-\lambda_i} \right) (-1)^{|l/\lambda|} f^{l/m}, \label{e.euler-two-pointed-phi}
\end{align}
since $|\lambda/\mu|+\rho=g$.  (Recall that the sums are over sequences $l$ and $m$ such that $l_i\geq \lambda_i$ and $m_i\leq \mu_i$ for all $i$.)

When the determinant
\[
\frac{1}{|l/m|!}f^{l/m}= \left| \frac{1}{(l_i-m_j+j-i)!}\right|_{1\leq i <j\leq r+1}
\]
is nonzero, there is a unique permutation $w\in S_{r+1}$ acting on the columns of the matrix which sorts the entries across rows into decreasing order; equivalently, 
\begin{equation}
\label{eqn:i-M}
\mu^-_j:=m_{w(j)} +j-w(j) 
\end{equation}
defines a partition $\mu^-\subseteq\mu$, using a variation of Fact 2.  Then $f^{l/m} = (-1)^{\sgn(w)} \, f^{l/\mu^-}$, and
collecting terms gives
\begin{equation}
\label{eqn:euler-mu}
\chi\left(\mathcal{O}_{G^{\bm{a},\bm{b}}_d(C,P,Q)}\right)
= \sum_{\mu^-\subseteq\mu}  \alpha^{\mu/\mu^-} \left( \prod_{i=1}^{r+1}\binom{l_i-1}{l_i-\lambda_i} \right) \,(-1)^{|l/\lambda|} f^{l/\mu^-} ,
\end{equation}
where the sum is over partitions $\mu^-\subseteq\mu$ and sequences $l=(l_1,\ldots,l_{r+1})$ of nonnegative integers such that $l_i\geq \lambda_i$ for all $i$, $|\mu/\mu^-| +|l/\lambda|=\rho$, and
\begin{align}
 \alpha^{\mu/\mu^-} &:= \sum_{w\in S_{r+1}}(-1)^{\sgn(w)} \left( \prod_{j=1}^{r+1} \binom{\mu_{w(j)}}
 {\mu_{w(j)}-\mu^-_{j}+j-w(j)}\right) \nonumber \\
 &=  \left| \binom{\mu_i}{\mu_i-\mu^-_j + j-i}\right|_{1\leq i,j\leq r+1}.   \label{eqn:M-formula}
\end{align}

Similarly, using Fact 2 again, when the determinant 
\[
\frac{1}{|l/\mu^-|!}f^{l/\mu^-}= \left| \frac{1}{(l_i-\mu^-_j+j-i)!}\right|_{1\leq i <j\leq r+1}
\]
is nonzero, there is a unique permutation $w\in S_{r+1}$ acting on the rows which sorts the entries into decreasing order down columns.  Equivalently, 
\begin{equation}
\label{eqn:i-L}
\lambda^+_i:= l_{w(i)}-w(i)+i,  \qquad\text{ for } 1\leq i\leq r+1,
\end{equation}
defines a partition $\lambda^+\supseteq \lambda$.  Then $f^{l/\mu^-} = (-1)^{\sgn(w)} \, {f^{\lambda^+/\mu^-}}$.  Collecting terms gives
\begin{equation}
\label{eqn:euler-lambda}
\left( \prod_{i=1}^{r+1}\binom{l_i-1}{l_i-\lambda_i} \right) \, f^{l/\mu^-} = \sum_{\lambda^+\supseteq \lambda}  \, \zeta^{\lambda^+/\lambda}  \cdot f^{\lambda^+/\mu^-} ,
\end{equation}
where the sum is over partitions $\lambda^+$ of length  $r+1$ (so $\lambda^+/\mu^-$ is a skew diagram) such that $|\lambda^+/\lambda|=|l/\lambda|$, and 
\begin{align}
 \zeta^{\lambda^+/\lambda} &:= \sum_{w\in S_{r+1}}(-1)^{\sgn(w)} 
 \left( \prod_{i=1}^{r+1} \binom{\lambda^+_{i}+w(i)-i-1}
 {\lambda^+_{i}-\lambda_{w(i)}+w(i)-i}\right) \nonumber \\
  &= \left| \binom{\lambda^+_i+j-i-1}{\lambda^+_i-\lambda_j+j-i}\right|_{1\leq i,j\leq r+1}.   \label{eqn:L-formula}
\end{align}

The binomial determinants $\alpha^{\mu/\mu^-}$ and $\zeta^{\lambda^+/\lambda}$ enumerate tableaux, by the method of Gessel-Viennot.  A \emph{(column) semi-standard Young tableau} on a given shape is a filling of the boxes by positive integers such that the entries are weakly increasing across each row and strictly increasing down each column.  A filling is a \emph{row semi-standard Young tableau} if the transpose condition holds: the entries are strictly increasing across each row and weakly increasing down each column.  A \emph{strict Young tableau} is a filling whose entries are strictly increasing across each row and down each column. 

By \cite[Theorem 14]{gv}, the determinant $\alpha^{\mu/\mu^-}$ is equal to the number of row semi-standard Young tableaux on $\mu/\mu^-$ whose entries in row $i$ are between $1$ and $\mu_i$, inclusive, and the determinant $\zeta^{\lambda^+/\lambda}$ is equal to the number of semi-standard Young tableaux on $\lambda^+/\lambda$ whose entries in row $i$ are between $-\lambda_i$ and $-1$, inclusive.  Such tableaux are in bijection with strict Young tableaux on $\lambda^+/\lambda$ whose entries in row $i$ are between $1$ and $\lambda^+_i-1$: given a semi-standard tableau on $\lambda^+/\lambda$ with $i$-th row entries in $\{-\lambda_i,\ldots,-1\}$, add to each entry the index of its column to obtain a strict tableau with $i$-th row entries in $\{1,\ldots,\lambda^+_i-1\}$.  Combining equations \eqref{eqn:euler-mu} and \eqref{eqn:euler-lambda} concludes the proof of Theorem~\ref{t.euler-two-pointed-tableau}.
\end{proof}

Next we will prove a nonemptiness criterion for the variety $G^{\bm{a},\bm{b}}_d(C, P, Q)$, using Corollary \ref{c.nonempty}.  By setting $\beta=0$ in \eqref{BNinCK}, and passing to numerical equivalence, we obtain a variation on the formula for the cohomology class of $W^r_d(C)$:

\begin{proposition}
\label{numclassW}
Let $(C,P,Q)$ be a smooth two-pointed curve of genus $g$.  If $\lambda_i\geq \mu_i$ for all $i$ and $W^{\ba,\bb}_d(C,P,Q)$ has dimension equal to $\rho$, then its numerical class is
\[
\left[W^{\ba,\bb}_d(C,P,Q)\right] = \left|  \frac{1}{(a_{r-i} + b_{j} +g -d)!} \right|_{0\leq i,j\leq r} 
\theta^{g-\rho}.
\]

If $G^{\ba,\bb}_d(C,P,Q)$ has dimension equal to $\rho$, then $\pi_*[G^{\ba,\bb}_d(C,P,Q)]$ equals $[W^{\ba,\bb}_d(C,P,Q)]$ when $\lambda/\mu$ is a skew diagram, and vanishes otherwise. 
\end{proposition}

In comparing with \eqref{BNinCK}, note the shift of indexing of the matrix, and recall that the definitions of $\lambda$ and $\mu$ imply $\lambda_i-\mu_j+j-i = a_{r+1-i} + b_{j-1} +g -d$.  The vanishing statement follows algebraically from Fact 1, or geometrically from the fact that $\dim \pi(G^{\ba,\bb}_d(C,P,Q))< \dim (G^{\ba,\bb}_d(C,P,Q))$, unless $\lambda/\mu$ is a skew diagram; see also the specialization of Proposition \ref{prop:aa'GW} at $\beta=0$.

Now we can state the nonemptiness criterion.

\pagebreak[3]

\begin{proposition}
\label{non-emptiness}
Let $(C,P,Q)$ be a smooth two-pointed curve of genus $g$.  If
\begin{align}
\label{oss}
\rho' := g - \sum_{i=0}^r \max\{0, a_i+b_{r-i} +g-d \} \geq 0,
\end{align}
\nopagebreak
then the locus of special linear series $G^{\ba,\bb}_d(C, P, Q)$ is non-empty.
\end{proposition}

\pagebreak[3]

\noindent This  was first proved by Osserman, using degeneration techniques \cite{o}.  When $\bb=(0,1,\dots,r)$, it recovers the statement for the one-pointed case in \cite[Proposition 1.2]{eh}.

\begin{proof}
The nonemptiness of $G^{\bm{a}, \bm{b}}_d(C,P,Q)$ is equivalent to the nonemptiness of its image $W=W^{\bm{a}', \bm{b}}_d(C,P,Q)$ in ${\rm Pic}^d(C)$.  By Corollary~\ref{c.nonempty}, $W$ is nonempty when the class $\Delta_{\lambda'/\mu}(c;0)$ is nonzero.  By Proposition \ref{numclassW}, this class is numerically equivalent to
\begin{equation}\label{e.nonemptypf}
 \left|   \frac{1}{(a'_{r-i} + b_{j} +g -d)!} \right|_{0\leq i,j\leq r} 
\theta^{g-\rho'},
\end{equation}
where, as before, $\ba'$ is the sequence defined by
\[
a'_i := a_i+ \max\{0,  d -g -a_i -b_{r-i}\}.
\]
This means $\rho' = \rho(g,r,d,\ba',\bb)$.

Associating partitions $\lambda'$ and $\mu$ to the data $g,d,\ba',\bb$ as usual, the definition of $\ba'$ guarantees that $\lambda'_i\geq \mu_i$ for all $i$.  It follows that the determinantal coefficient is nonzero (see Fact 1), so the expression \eqref{e.nonemptypf} is nonzero if and only if $\rho'\geq 0$.  This is equivalent to the condition in \eqref{oss}. 
\end{proof}

Now we turn to some special cases.

\subsection{The curve case}
\label{rho1} 
Let us write $\lambda+\epsilon_i$ for the diagram obtained by adding one box to the right of the $i$-th row of $\lambda$, and $\mu-\epsilon_i$ for the diagram obtained by subtracting one box from the $i$-th row of $\mu$.  (This means the diagram $\lambda/(\mu-\epsilon_i)$ is obtained by {\em adding} one box to the left of the $i$-th row of $\lambda/\mu$.)

Now assume $\rho(g,r,d,\bm{a},\bm{b})=1$.  The reformulation in (\ref{e.euler-two-pointed-phi}) of Theorem \ref{t.BN} reduces to
\begin{align*}
\chi \left(\OO_{G^{\bm{a},\bm{b}}_d(C, P, Q)} \right) =
 \sum_{i=1}^{r+1} \mu_i\, f^{\lambda/(\mu-\epsilon_i)}
-\sum_{i=1}^{r+1} \lambda_i\, f^{(\lambda+\epsilon_i)/\mu}.
\end{align*}
By Fact 1, $f^{\lambda/(\mu-\epsilon_i)}$ vanishes when $\lambda/(\mu-\epsilon_i)$ is not a skew diagram, and $f^{(\lambda+\epsilon_i)/\mu}$ vanishes when $(\lambda+\epsilon_i)/\mu$ is not a skew diagram.  Using the identity 
\[
(r+1) (|\lambda/\mu|+1) f^{\lambda/\mu} = \sum_{i=1}^{r+1} (\lambda_i + r+2-i) f^{(\lambda+\epsilon_i)/\mu}  - \sum_{i=1}^{r+1} (\mu_i +r+1-i) f^{\lambda/(\mu-\epsilon_i)},
\]
for the number of standard skew Young tableaux, we recover \cite[Theorem 1.2]{clpt}.

\subsection{The one-pointed case} 
\label{sec:1ptcase}
When $\bm{b}=(0,\dots,r)$, the locus $G^{\bm{a},\bm{b}}_d(C, P, Q)$ is identical to the one-pointed locus 
\[
G^{\bm{a}}_d(C, P) := \left\{ \ell \in G^r_d(C) \, | \, \bm{a}^\ell(P) \geq \bm{a}\right\}.
\]
On the other hand, when the points $P$ and $Q$ collide together on the curve $C$, the locus $G^{\bm{a},\bm{b}}_d(C, P, Q)$ specializes to the locus of linear series $(L,V)\in G^r_d(C)$ such that 
\[
\dim \left( V\cap H^0(L(-(a_i + b_{r-j})P))\right)\geq 1+j-i.
\]
Fix $l$ such that $l_i\geq \lambda_i$ and $|l/\lambda|=\rho$.  By an application of the Vandermonde identity,
\[
 \left|  \frac{\theta^{l_i+j-i}}{(l_i+j-i)!} \right|_{1\leq i,j\leq r+1} 
= g! \frac{\prod_{1\leq i<j \leq r+1} (l_i-l_j+j-i)}{\prod_{i=1}^{r+1} (l_i+r+1-i)!},
\]
so Theorem~\ref{t.BN} reduces to
\[
\chi \left(\OO_{G^{\bm{a}}_d(C, P)}\right)
= \sum_{|l/\lambda|=\rho} 
{-\lambda_1 \choose l_1-\lambda_1} \cdots {-\lambda_{r+1} \choose l_{r+1}-\lambda_{r+1}} 
\times
g! \frac{\prod_{1\leq i<j \leq r+1} (l_i-l_j+j-i)}{\prod_{i=1}^{r+1} (l_i + r+1-i)!}.
\]
When in addition $\rho(g,r,d,\bm{a},\bm{b})=1$, this sum becomes
\[
\chi \left(\OO_{G^{\bm{a}}_d(C, P)}\right) = - g! \sum_{k=0}^r (g-d+r+a_k-k) \frac{\prod_{0\leq i<j\leq r} (a_j - a_i + \delta^k_j - \delta^k_i)}{\prod_{i=0}^r (g-d+r+a_i+\delta^k_i)!},
\]
where $\delta$ is the Kronecker delta.

\subsection{Set-valued tableaux and the one-pointed case} \label{svt-one-pointed}
In the one-pointed case, we can re-write the Euler characteristic in terms of numbers of certain tableaux. A \textit{set-valued tableau} on a  shape $\lambda$ is a labelling of the boxes of $\lambda$ by finite non-empty subsets of $\mathbb{N}$ such that the maximum element of the label of any box $(i,j)$ is at most the minimum element of the label at $(i,j+1)$, and smaller than the minimum element of the label at $(i+1,j)$ (see \cite{b}).  Given a nonnegative integer $\rho$, a \emph{$\rho$-standard set-valued tab\-leau} on $\lambda$ is a set-valued tableau on $\lambda$ such that the labels of the boxes of $\lambda$ are subsets of $\left\{1,\ldots,|\lambda|+\rho\right\}$ and each of $1,\ldots,|\lambda|+\rho$ appears exactly once.  (See \S\ref{Schub} for more about set-valued tableaux and the connection with Grothendieck polynomials.)

Chan and Pflueger conjectured a formula expressing the Euler characteristic of a two-pointed Brill-Noether locus via set-valued tableaux on a skew shape.  The following establishes the one-pointed version of their conjecture.\footnote{Chan and Pflueger have now proved their  conjecture using different methods in \cite{cp}.}

\begin{corollary}
\label{cor:CPconj}
Suppose $\dim G^{\bm{a}}_d(C,P)= \rho$, and let $\lambda$ be the partition corresponding to $\ba$.  Then 
\[
\chi\left(\mathcal{O}_{G^{\bm{a}}_d(C,P)}\right) = (-1)^\rho\cdot \# \{ \text{$\rho$-standard set-valued tableaux on $\lambda$}\}.
\] 
This is zero if and only if $W^{\bm{a}}_d(C,P)={\rm Pic}^d(C)$.
\end{corollary}

\begin{proof} 
In the one-pointed case, Theorem \ref{t.euler-two-pointed-tableau} becomes
\[
\chi\left(\mathcal{O}_{G^{\bm{a}}_d(C,P)}\right)= (-1)^\rho \sum_{|\lambda^+/\lambda|=\rho}\zeta^{\lambda^+/\lambda} \cdot f^{\lambda^+},
\]
so we must identify the sum on the RHS with the number of $\rho$-standard set-valued tableaux on $\lambda$.

For any partition $\nu$, it follows from a theorem of Lenart  \cite[Theorem 2.2]{lenart} that\footnote{The result of Lenart has now been extended to skew shapes in \cite[Theorem 6.8]{cp}. By means of this extension, the argument in the proof of Corollary \ref{cor:CPconj} can be applied to prove the two-pointed version of the statement.
}
\[ 
\# \{ \text{$\rho$-standard set-valued tableaux on $\nu$}\}= \sum_{|\nu^+/\nu|=\rho}  g^{\nu^+/\nu} f^{\nu^+},
\]
where $g^{\nu^+/\nu}$ is the number of strict Young tableaux on $\nu^+/\nu$ whose entries in row $i$ are between $1$ and $i-1$, inclusive, and $f^{\nu^+}$ is the number of standard Young tableaux on $\nu^+$.  (To deduce this from Lenart's theorem, which writes the Grothendieck polynomial for $\nu$ as a sum of Schur polynomials, compare the coefficient of the monomial $x_1\cdots x_{|\nu|+\rho}$ on each side of his formula.)

Our claim follows by taking $\nu$ to be the conjugate partition $\lambda'$, i.e., the diagram obtained by reflecting across the diagonal, so that rows and columns are interchanged.  It is easy to see that $\rho$-standard set valued tableaux on $\lambda$ and $\nu$ are in bijection.  Similarly, standard Young tableaux on $\lambda^+$ and $\nu^+=(\lambda^+)'$ are also in bijection, so $f^{\lambda^+}=f^{\nu^+}$.  Finally, $\zeta^{\lambda^+/\lambda} = g^{\nu^+/\nu}$, for $\nu=\lambda'$ and $\nu^+=(\lambda^+)'$, because sending a tableau $T$ to its conjugate $T'$ defines a bijection from strict tableaux on $\lambda^+/\lambda$ with $i$-th row entries in $\{1,\ldots,\lambda^+_i-1\}$ to strict tableaux on $\nu^+/\nu$ with $i$-th row entries in $\{1,\ldots,i-1\}$.
\end{proof}

\subsection{The classical case}

Here there are no point conditions, and in the formulas one can take $\mu=\emptyset$, and let $\lambda = (g-d+r)^{r+1}$ be the rectangular shape.  Any partition $\lambda^+\supseteq \lambda$ of length $r+1$ can be written as $\lambda+\gamma$, for some partition $\gamma$ of length $r+1$.  The determinant $\zeta^{\lambda^+/\lambda}$ can therefore be written as
\begin{align*}
\zeta^{\lambda^+/\lambda} 
&=  \left| \frac{(g-d+r+\gamma_i+j-i-1)_{\gamma_i+j-i}}{(\gamma_i+j-i)!}\right|_{1\leq i,j\leq r+1},
\end{align*}
where $(n)_k = n(n-1)\cdots(n+1-k)$ is the falling factorial.  Manipulating the matrix leads to a factorization of this determinant as
\[
\left| \frac{1}{(\gamma_i+j-i)!}\right|_{1\leq i,j\leq r+1}\cdot \prod_{i=1}^{r+1} (g-d+r+\gamma_i-i)_{\gamma_i}.
\]
Applying this simplification to Theorem \ref{t.euler-two-pointed-tableau}, we obtain:
\begin{corollary} 
\label{c.euler-classical}
If $\dim G^r_d(C) = \rho(g,r,d)\leq g$, then 
\begin{align*}
\chi \left(\mathcal{O}_{G^r_d(C)}\right)&=
\chi \left(\mathcal{O}_{W^r_d(C)}\right) \\
&=\frac{(-1)^\rho}{\rho !} \sum_{|\gamma|=\rho} f^\gamma  \cdot \prod_{i=1}^{r+1} (g-d+r+\gamma_i-i)_{\gamma_i} \cdot f^{\lambda+\gamma} 
\end{align*}
where the sum is over partitions $\gamma=(\gamma_1\geq\cdots\geq \gamma_{r+1}\geq 0)$, and $\lambda+\gamma$ is the partition $(g-d+r+\gamma_1,\allowbreak\ldots,g-d+r+\gamma_{r+1})$.
\end{corollary}

In low dimensions, the Euler characteristic can be written in a fairly simple closed form.

When $\rho(g,r,d)=0$, the formula in Corollary \ref{c.euler-classical} recovers Castelnuovo's count for the number of line bundles of degree $d$ with $r+1$ sections:
\[
N^r_{g,d}:= \chi \left(\OO_{G^r_d(C)}\right) = g!\prod_{i=0}^r \frac{i!}{(g-d+r+i)!}.
\]
When $\rho(g,r,d)=1$, we recover \cite[Theorem 4]{eh}:
\[
\chi \left(\OO_{G^r_d(C)}\right) = -\frac{(g-d+r)(r+1)}{g-d+2r+1} N^r_{g,d}.
\]
When $\rho(g,r,d)=2$, we obtain
\[
\chi \left(\OO_{G^r_d(C)}\right) = \frac{(r+1)^2(g-d+r)^2}{2(g-d+2r)(g-d+2r+2)} \,N^r_{g,d}.
\]
Finally, when $\rho(g,r,d)=3$, set $s:=g-d+r$, and  we have
\[
\chi \left(\OO_{G^r_d(C)}\right) = - \frac{(r+1)^2 s^2 \left[ \left((r + s + 1)^2 - 2 \right) s (r + 1) - 2 \right]}{6(s+r-1)(s+r)(s+r+1)(s+r+2)(s+r+3)} N^r_{g,d}.
\]

\section{Schubert and Grothendieck polynomials}
\label{Schub}

As one more application of the degeneracy locus formula of Theorem~\ref{t.degloci}, we deduce determinantal formulas for \textit{(double) Schubert} and \textit{Grothendieck polynomials} for 321-avoiding permutations. Indeed in this section, we identify our K-theory formulas with \textit{double} versions of the {\em flagged skew Grothendieck polynomials} recently introduced by Matsumura \cite{matsumura}.

For decreasing sequences $\bp =(p_1,\dots,p_t)$ and $\bq=(q_1,\dots,q_t)$, we defined partitions $\lambda$ and $\mu$  by
\begin{align*}
  \lambda_i &= q_i-t+i, &
  \mu_j &= p_{j}-(t+1-j)
\end{align*}
in \S\ref{ss.321}.
When $\bp$ and $\bq$ satisfy
\renewcommand{\theequation}{$*$}
\begin{equation}
 q_i \geq p_{i} -1 \quad \text{ for all }i,
\end{equation}
the partitions form a skew diagram $\lambda/\mu$, and we defined an associated permutation $w$ by setting
\[
  w(p_i) = q_i+1 \qquad\mbox{for $1\leq i\leq t$,}
\]
 and then filling in the remaining entries with the unused numbers in increasing order.  
As noted in \S\ref{ss.321}, this is a 321-avoiding permutation, and all 321-avoiding permutations arise this way.

\begin{remark} 
\label{r.labeled-skew} The above is equivalent to the bijection of 321-avoiding permutations with labeled skew tableaux of Billey-Jockusch-Stanley (\cite{bjs}), which can be re-formulated as follows. For a 321-avoiding permutation $w$, the skew shape $\sigma(w)$ considered in \cite{bjs} is a $180$ degree rotation of our skew shape $\lambda/\mu$, that is, $\sigma(w)=\eta/\tau$ where $\eta_i=\lambda_1-\mu_{t+1-i}$ and $\tau_i=\lambda_1-\lambda_{t+1-i}$.  Let $f_w=(f_1,f_2,\ldots,f_t)$ be the increasing sequence of indices $j$ such that $w(j)>j$, and let $e_i=w(f_i)-1$. Then the labeling $\omega(w)$ of the skew shape $\sigma(w)$ is obtained by placing the entries $e_i,e_i-1,\ldots, f_i$ in the $i$-th row of $\sigma(w)$ such that the entries increase by one in each column, and decrease by one in each row. In our setup, the labeling $\omega(w)$ is determined by $f_i=p_{t+1-i}$ and $e_i=q_{t+1-i}$.
\end{remark}

\renewcommand{\theequation}{\arabic{equation}}
\addtocounter{equation}{-1}

\subsection{Schubert polynomials}

For any permutation $w$, the \textit{double Schubert polynomial} $\mathfrak{S}_w(x,y)$ of Lascoux and Sch\"utzenberger is a canonical representative  for the cohomology class of the corresponding Schubert variety or degeneracy locus \cite{fulton-flags}.  Similarly, the \textit{double Grothendieck polynomial} $\mathfrak{G}_w(x,y)$ represents the structure sheaf of a Schubert variety or degeneracy locus in K-theory \cite[Theorem~3]{fl}.  
These polynomials are defined inductively, but for special types of permutations, one can give direct formulas.  Our goal here is to give such formulas for 321-avoiding permutations.

First we state the formula for Schubert polynomials.  Given sets of variables $x$ and $y$, let
\[
  c(i,j) = \frac{\prod_{a=1}^{q_i} (1-u y_a)}{\prod_{b=1}^{p_{j}} (1-u x_b)},
\]
and define $c_k(i,j)$ by collecting the coefficient of $u^k$ in the expansion of this rational function (in positive powers of $x$ and $y$).  
For example, if $y=0$, then $c_k(i,j)$ is the complete homogeneous symmetric polynomial $h_k(x_1,\ldots,x_{p_{j}})$ (for any $i$), and if $x=0$, then $c_k(i,j)$ is the elementary symmetric polynomial $(-1)^k e_k(y_1,\ldots,y_{q_i})$ (for any $j$).

\begin{corollary}
\label{c.double-schubert}
Let $w$ be a 321-avoiding permutation, with associated tuples $\bp,\bq$ satisfying \eqref{e.skew}, and let $\lambda/\mu$ be the corresponding skew Young diagram.  The double Schubert polynomial for $w$ has the following determinantal expression:
\[
  \mathfrak{S}_w(x,y) = \Delta_{\lambda/\mu}(c;0) = \left| c_{\lambda_i-\mu_j+j-i}(i,j) \right|_{1\leq i,j\leq t},
\]
where $c_k(i,j)$ is the polynomial in $x$ and $y$ defined above.
\end{corollary}

\noindent
Since double Schubert polynomials are obtained by specializing double Grothendieck polynomials at $\beta=0$, the statement is a special case of Corollary~\ref{c.321-double-groth}, proved below.

This recovers a formula of Lascoux and Chen-Yan-Yang (see \cite{cyy}), which in turn generalized a formula of Billey-Jockusch-Stanley \cite{bjs} for the single Schubert polynomials of 321-avoiding permutations --- that is, the case $y=0$.  
More precisely, the matrices computing these formulas in \cite{cyy} are obtained by reflecting about the anti-diagonal the matrices computing the determinants in Corollary \ref{c.double-schubert}.  
The right-hand side is a {\em flagged double skew Schur function}, a variant of the flagged double Schur function introduced by Chen-Li-Louck \cite{cll}.  (``Flagging'' refers to the nested sets of variables appearing along rows and columns of the determinant: the $i$-th row uses $\{y_1,\ldots,y_{q_i}\}$, and the $j$-th column uses $\{x_1,\ldots,x_{p_j}\}$.) 

\begin{example}
\label{ex:31254}
An example of a 321-avoiding permutation which is not also vexillary (another class having determinantal expressions, thanks to an older theorem of Wachs) is $w=3\;1\;2\;5\;4$.  Here $\bp=(4,1)$ and $\bq=(4,2)$, so $\lambda = (3,2)$ and $\mu=(2,0)$, and the formula says
\begin{align*}
  \mathfrak{S}_{31254} &= \left|\begin{array}{cc} c_1(1,1) & c_4(1,2) \\ 0 & c_2(2,2) \end{array}\right| = c_1(1,1)\cdot c_2(2,2) \\
   &= (x_1+x_2+x_3+x_4-y_1-y_2-y_3-y_4)\!\cdot\!(x_1^2-x_1y_1-x_1y_2+y_1y_2).
\end{align*}

Comparing with \cite{bjs}, and using their notation, the labeled skew diagram $(\sigma(w)=\eta/\tau,\omega(w))$ associated to $w=3\;1\;2\;5\;4$ is given by:
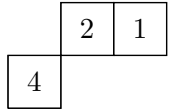
\[
\ytableausetup{mathmode, boxsize=2em}
\begin{ytableau}
\none & 2 & 1\\
4& \none & \none
\end{ytableau}
\]
where $\eta=(3,1)$, $\tau=(1,0)$, $f=(1,4)$, and $e=(2,4)$. The matrix computing the determinant $  \mathfrak{S}_{31254}$ in \cite{bjs} is obtained by reflecting the above matrix  about the anti-diagonal.
\end{example}

\subsection{Grothendieck polynomials}

Now we turn to Grothendieck polynomials.  Here the variables should be identified as follows.  
Let
\[
  c(i,j) = \prod_{a=1}^{q_i}\prod_{b=1}^{p_j} \frac{(1+\beta y_a-u y_a)(1+\beta x_b)}{(1+\beta y_a)(1+\beta x_b-u x_b)}.
\]
The term $c_k(i,j)$ is obtained as before, by expanding and collecting the coefficient of $u^k$.

\begin{corollary}\label{c.321-double-groth}
Let $w$ be a 321-avoiding permutation, with associated tuples $\bp,\bq$ satisfying \eqref{e.skew}, and let $\lambda/\mu$ be the corresponding skew Young diagram.  The double Grothendieck polynomial for $w$ has the following determinantal expression:
\begin{align*}
  \mathfrak{G}_w(x,y) &= \Delta_{\lambda/\mu}(c;\beta) \\
  &=\left| \sum_{k\geq 0} \binom{\lambda_i-\mu_j+k-1}{k}\beta^k c_{\lambda_i-\mu_j+j-i+k}(i,j) \right|_{1\leq i,j\leq t},
\end{align*}
where $c_k(i,j)$ is the polynomial in $x$, $y$, and $\beta$ defined above.
\end{corollary}

\begin{proof}
This follows directly from Theorem~\ref{t.deg}(\ref{t.deg-part2}), by choosing a base and vector bundles so that there are no relations among the relevant Chern classes.

Here is one way to do this.  Let $\Fl(\bp,V)$ and $\Fl(V,\bq)$ be the partial flag varieties of  subspaces of dimesions $p_j$ and quotients of dimensions $q_i$, respectively, so they come with tautological bundles $E_{p_j}\subseteq V$ and $V \twoheadrightarrow F_{q_i}$.  Let $X=\Fl(\bp,V)\times \Fl(V,\bq)$, and identify variables $x$ and $y$ with Chern classes of the tautological bundles by writing
\[
 c(E_{p_j})=\prod_{b=1}^{p_j} \frac{1+\beta x_b - ux_b}{1+\beta x_b} \quad\mbox{and} \quad c(F_{q_i})=\prod_{a=1}^{q_i} \frac{1+\beta y_a - uy_a}{1+\beta y_a};
\]
that is, the $x$ variables are the Chern roots of $E_{p_j}^*$, the $y$ variables are the Chern roots of $F_{q_i}^*$, and we have $c(i,j) = c(F_{q_i}-E_{p_j})$.  
For any fixed degree $d$, one can take $\dim V$ sufficiently large so that there are no relations among the Chern classes of $E_{p_j}$ and $F_{q_i}$ in $CK^d(X)$.

Via the projection $X \to \Fl(V,\bq)$, one can regard $X$ as a (partial) flag bundle.  By Lemma~\ref{l.w}, the degeneracy locus $W_{\bp,\bq} \subseteq X$, defined by the conditions
\[
\dim \textrm{ker} \left(E_{p_j}\rightarrow F_{q_i} \right)\geq 1+i-j \qquad \mbox{for all $i,j$},
\]
is identified with the Schubert locus corresponding to $w$ in this flag bundle.  Now \cite[Theorem~3]{fl} says this locus is represented by the double Grothendieck polynomial $\mathfrak{G}_w(x,y)$ in K-theory, while Theorem~\ref{t.deg}(\ref{t.deg-part2}) says it is represented by $\Delta_{\lambda/\mu}(c;\beta)$.  Since there are no relations among the variables, we must have an equality of polynomials.
\end{proof}

We conclude with a tableau formula for the Grothendieck polynomials  $\mathfrak{G}_w(x) =\mathfrak{G}_w(x,0)$. As in Section \ref{svt-one-pointed},
a \textit{set-valued tableau} of skew shape $\lambda/\mu$ is a labelling of the boxes of $\lambda/\mu$ by finite non-empty subsets of $\mathbb{N}$ such that the maximum element of the label of any box $(i,j)$ is at most the minimum element of the label at $(i,j+1)$, and smaller than the minimum element of the label at $(i+1,j)$ (see \cite{b}).  Given a skew shape $\lambda/\mu$ and a flagging $f = (f_1, \ldots, f_t)$, a \emph{flagged skew set-valued tableau} of skew shape $\lambda/\mu$ with flagging $f$ is a set-valued tableau on $\lambda/\mu$ such that every entry in the $i$-th row is a subset of $\{1,2,\ldots,f_i\}$.  Let $FSVT(\lambda/\mu,f)$ denote the set of all such flagged skew tableaux. For a (flagged) set-valued tableau $T$, let  $x^T$ be the monomial in which the exponent of $x_i$ is the number of boxes of $T$ which contain $i$. 

\begin{corollary}
\label{c.m}
Let $w$ be a 321-avoiding permutation and let $\sigma(w)=\eta/\tau$ be the skew Young diagram with flagging $f_w$ corresponding to $w$ via the Billey-Jockusch-Stanley bijection, as in Remark~\ref{r.labeled-skew}.  The Grothendieck polynomial $\mathfrak{G}_w(x) =\mathfrak{G}_w(x,0)$ is equal to 
\begin{equation} 
\label{eqn:fsgp}
\sum_{T\in FSVT(\sigma(w),f_w)} \beta^{|T|-|\sigma(w)|} x^T.
\end{equation}
\end{corollary}

\begin{proof}
Matsumura \cite[\S4]{matsumura} defined \textit{flagged skew Grothendieck polynomials} to be generating functions of flagged set-valued tableaux given by  \eqref{eqn:fsgp}, and proved that they have determinantal expressions. 
Corollary~\ref{c.321-double-groth} also holds after the matrix is reflected about the anti-diagonal --- by replacing the $(i,j)$ entry with the $(t+1-j,t+1-i)$ entry --- since the determinant is unchanged by this operation.  The entries of this reflected matrix are equal to those in the determinantal formulas of \cite[\S4]{matsumura}, as explained in \cite[Remark~1.1]{a}.
\end{proof}

\subsection{Flagged set-valued skew tableaux and pipe dreams}

Corollary \ref{c.m} recovers the tableau formulas for Schubert polynomials of 321-avoiding permutations \cite[Theorem 2.2]{bjs} and for Grothendieck polynomials of Grassmann permutations \cite[Theorem 5.8]{kmy}.  The proofs of those formulas rely on writing Schubert and Grothendieck polynomials in terms of \emph{pipe dreams} (after \cite{km}) along with bijections between certain tableaux and pipe dreams (\cite[Theorem~2.2]{bjs}, \cite[Proposition~5.3]{kmy}).  To conclude, we give a bijection between flagged set-valued skew tableaux and pipe dreams for 321-avoiding permutations that extends these bijections and gives an alternative proof of Corollary \ref{c.m}.

For a flagged set-valued skew tableau $T\in FSVT(\sigma(w),f_w)$, let $\overline{T}$ be the flagged skew tableau obtained by taking the smallest element of each box of $T$ (see Figure \ref{fig:T}). 

\begin{figure}[htb]
\begin{center}
\begin{tabular}{cc} 
$T=$ 
\ytableausetup{mathmode, boxsize=2em}
\begin{ytableau}
\none & 1 & 1\\
23& \none & \none
\end{ytableau}
& \qquad\qquad
$\overline{T}=$ 
\ytableausetup{mathmode, boxsize=2em}
\begin{ytableau}
\none & 1 & 1\\
2& \none & \none
\end{ytableau}
\end{tabular}
\end{center}
\caption{For the permutation $w=31254$ in Example \ref{ex:31254}, and the given $T\in FSVT(\sigma(w),f_w)$, $\overline{T}$ is the associated flagged skew tableau.}
\label{fig:T}
\end{figure}

A \emph{pipe} dream is  a tiling of the fourth quadrant of the plane  by  \emph{crosses} $\textcross$ and \emph{elbows} $\textelbow$. A \emph{reduced pipe dream} (or rc-graph) for a  permutation $w$  is a tiling such that the pipe that starts at the beginning of the $i$-th row exits the top of the $w_i$-th column, with no two pipes of $P$ crossing each other more than once.  For a pipe dream $P$, we write $\overline{P}$ for the reduced pipe dream obtained by replacing all but the northeasternmost cross between two pipes by an elbow, and say that $P$ is a pipe dream for a permutation $w$ if $\overline{P}$ is a reduced pipe dream for $w$.

\begin{figure}[htb]
\begin{center}
\begin{tabular}{cc} 
 $P=  
\begin{array}{cccccc}
	& \perm1{} & \perm2{} &  \perm3{} &  \perm4{} &  \perm5{}\\
 		\petit1 & \+ & \+& \jr &  \jr & \jr \\
		\petit2 &   \jr & \jr & \+ & \jr  &  \jr   \\
		\petit3 & \jr  &  \+ & \jr & \jr  & \jr   \\
		\petit4 & \jr & \jr   & \jr &  \jr & \jr  \\
		\petit5 & \jr  &\jr  & \jr   &  \jr  & \jr
\end{array}$
&
 $\qquad\overline{P}=  \begin{array}{cccccc}
	& \perm1{} & \perm2{} &  \perm3{} &  \perm4{} &  \perm5{}\\
 		\petit1 & \+ & \+ & \jr & \jr&  \jr  \\
		\petit2 &   \jr & \jr & \+ & \jr  &  \jr   \\
		\petit3 & \jr  & \jr & \jr & \jr  & \jr   \\
		\petit4 & \jr & \jr   & \jr &  \jr & \jr  \\
		\petit5 & \jr  &\jr  & \jr   &  \jr  & \jr  
\end{array}$
\end{tabular}
\end{center}
\caption{A pipe dream $P$ and related reduced pipe dream  $\overline{P}$ for the permutation $w=31254$.}
\label{fig:P}
\end{figure}

Following the notation of Remark \ref{r.labeled-skew}, to $T\in FSVT(\sigma(w),f_w)$ we can associate a pipe dream $\Omega(T)$ as follows: 
\begin{quote} 
Place a cross in position
 $(i,\omega(b)-i+1)$ for each entry $i$ in a box $b$ of $\sigma(w)$, and an elbow in all other positions.
\end{quote}

Given a pipe dream $P$, let $x^P:=\prod_{(i,j)}x_i$, where the product is over crosses $(i,j)$ in $P$, and let $|P|$ be the total number of crosses. Fomin and Kirillov (\cite{fk}, \cite{fk2}; see also \cite{km} for the language of pipe dreams) show that 
\begin{equation}
\label{eqn:Gpipe}
 \mathfrak{G}_w(x) = \sum_P \beta^{|P|-\ell(w)}x^P
 \end{equation}
where the sum is over $P$ such that $\overline{P}$ is a reduced pipe dream for $w$. This specializes to \cite[Theorem 1.1]{bjs} when $\beta=0$. 

\begin{example}
For $T$ as in Figure \ref{fig:T}, we have crosses in exactly positions $(1,1),(1,2),(2,3)$, and $(3,2)$, so that $\Omega(T)$ is equal to the pipe dream $P$ of Figure~\ref{fig:P}. Similarly,  for $\overline{T}$ as in Figure~\ref{fig:T},  $\Omega(\overline{T})$ is equal to the reduced pipe dream $\overline{P}$ of Figure \ref{fig:P}.   Here,  $x^{\overline{T}}=x_1^2x_2=x^{\overline{P}}$ and $x^T=x_1^2x_2x_3=x^P$, where $P=\Omega(T)$ and $\overline{P}=\Omega(\overline{T})$.
\end{example}

\begin{proposition} \label{prop:bijection} Let $w$ be a 321-avoiding permutation and let $\sigma(w)=\eta/\tau$ be the skew Young diagram with flagging $f_w$ corresponding to $w$. Then the map $\Omega$ gives a weight-preserving bijection from $FSVT(\sigma(w),f_w)$ to the set of pipe dreams for $w$.
\end{proposition}

The map $\Omega$ generalizes the bijecton between flagged skew tableaux and reduced pipe dreams of $321$-avoiding permutations in \cite[Theorem 2.2]{bjs} and the bijection between flagged set valued tableaux and pipe dreams of Grassmann permutations in \cite[Proposition 5.3]{kmy}.

\begin{proof}   
By its definition the map $\Omega$ is an injection and specializes to the bijection between flagged skew tableaux and reduced pipe dreams for $w$. Therefore, if $P$ is a pipe dream for $w$, there is a flagged skew tableau $\overline{T}$ such that $\Omega(\overline{T})=\overline{P}$.  Since the proof of  \cite[Proposition~5.3(a)]{kmy} for ordinary shapes carries through for skew shapes, no pipe of $\overline{P}$ passes horizontally through one cross and vertically through another. (For straight shapes, the pipe dreams in \cite{kmy} and here differ by a reflection across the vertical axis.) 

We claim that if a horizontal and vertical pipe cross at a $\textcross$ and pass through a $\textelbow$ tile southwest of it, then the two tiles lie on the same anti-diagonal. This holds since if a pipe crosses horizontally at position $(i_0, j_0)$ it cannot cross any pipe vertically, hence to the west of $(i_0, j_0)$, the pipe is bounded to be at or above the anti-diagonal through $(i_0, j_0$). Similarly, if a pipe crosses vertically at $(i_0, j_0)$, it cannot cross  any pipe horizontally, hence going south of $(i_0, j_0$), the pipe is bounded to be at or below the anti-diagonal through $(i_0, j_0)$. If the two pipes also meet at an elbow, then that  elbow must lie on the anti-diagonal containing $(i_0, j_0)$.

The pipe dream $P$ is obtained from $\overline{P}$ by altering some such $\quad\jr\quad$ tiles to $\textcross$ tiles. This corresponds exactly to inserting extra entries in the box of $\overline{T}$ corresponding to the original  $\textcross$ tile of $\overline{P}$ . More specifically, let $b$ be a box in the labelled skew diagram $\sigma(w)$.  Let $i_0$ be the smallest entry of $b$.  Then for $i>i_0$, the additional entries of $b$ correspond to crosses in positions $(i,\omega(b)-i+1)$ in the pipe dream $\Omega(T)$ (these are southwest of the entry $(i_0, \omega(b)-i_0+1)$  in the anti-diagonal $\{(i,j)\,|\, i+j=\omega(b)+1\}$), and conversely. 

The bijection in \cite{bjs} and \cite{bb} between flagged skew tableaux  and reduced pipe dreams satisfies  $x^{\overline{T}}=x^{\Omega(\overline{T})}$.  By the description of the extra entries in fillings of $\overline{T}$, we conclude that $x^T=x^{\Omega(T)}$ and so $\Omega$ is a weight-preserving bijection.
\end{proof}

We observe that  $|\sigma(w)|=|\overline{T}|=|\Omega(\overline{T})|=l(w)$.  Comparing the summands in  \eqref{eqn:fsgp} and  \eqref{eqn:Gpipe}
 under  the above bijection, this shows that the generating function formulas \eqref{eqn:fsgp} and  \eqref{eqn:Gpipe}
 agree term by term, and therefore, this bijection gives an alternative proof of Corollary \ref{c.m}.

\providecommand{\bysame}{\leavevmode\hbox to3em{\hrulefill}\thinspace}
\providecommand{\MR}{\relax\ifhmode\unskip\space\fi MR }
\providecommand{\MRhref}[2]{%
  \href{http://www.ams.org/mathscinet-getitem?mr=#1}{#2}
}
\providecommand{\href}[2]{#2}

\end{document}